\documentclass[a4paper]{amsart}

\usepackage[all]{xy}
\usepackage{amsmath}
\usepackage{amssymb}
\usepackage{amsthm}
\usepackage{latexsym}
\usepackage{enumerate}
\usepackage{prettyref}
\usepackage{hyperref}
\usepackage{bm}
\usepackage{mathrsfs}

\newtheorem{theorem}{Theorem}[section]
\newrefformat{thm}{\hyperref[{#1}]{Theorem~\ref*{#1}}}
\newtheorem{definition}[theorem]{Definition}
\newrefformat{def}{\hyperref[{#1}]{Definition~\ref*{#1}}}
\newtheorem{lemma}[theorem]{Lemma}
\newrefformat{lem}{\hyperref[{#1}]{Lemma~\ref*{#1}}}
\newtheorem{proposition}[theorem]{Proposition}
\newrefformat{prop}{\hyperref[{#1}]{Proposition~\ref*{#1}}}
\newtheorem{corollary}[theorem]{Corollary}
\newrefformat{cor}{\hyperref[{#1}]{Corollary~\ref*{#1}}}
\newtheorem{remark}[theorem]{Remark}
\newrefformat{rem}{\hyperref[{#1}]{Remark~\ref*{#1}}}
{
\newtheorem{examplecore}[theorem]{Example}}
\newrefformat{ex}{\hyperref[{#1}]{Example~\ref*{#1}}}

\newenvironment{example}{\begin{examplecore}}{\hspace*{\fill}
$\square$\par\vspace{.1cm}\end{examplecore}}

\newcommand{\op}{\operatorname}
\newcommand{\et}{\mathrm{\acute{e}t}}

\begin{document}

\title{On stably trivial spin torsors over low-dimensional schemes}  

\author{Matthias Wendt}

\date{2017}

\address{Matthias Wendt} 
\email{m.wendt.c@gmail.com}
\thanks{The project was initiated during a pleasant stay at Wuppertal financed by the DFG SPP 1786 ``Homotopy theory and algebraic geometry''; and it was finished during a pleasant stay at Institut Mittag-Leffler in the special program ``Algebro-geometric and homotopical methods''.}

\subjclass[2010]{14F42,19G05 (20G35)}
\keywords{quadratic forms, stably hyperbolic, $\mathbb{A}^1$-homotopy, obstruction theory}

\begin{abstract}
The paper discusses stably trivial torsors for spin and orthogonal groups over smooth affine schemes over infinite perfect fields of characteristic unequal to $2$. We give a complete description of all the invariants relevant for the classification of such objects over schemes of dimension at most $3$, along with many examples. The results are based on the $\mathbb{A}^1$-representability theorem for torsors and transfer of known computations of $\mathbb{A}^1$-homotopy sheaves along the sporadic isomorphisms to spin groups.
\end{abstract}

\maketitle
\setcounter{tocdepth}{1}
\tableofcontents

\section{Introduction}

The main point of the paper is to study the classification of stably trivial orthogonal (and spin) bundles over low-dimensional schemes. This is essentially the question how large the difference is between the Grothendieck--Witt ring and the actual isometry classification of quadratic forms over rings. Over fields, Witt's cancellation theorem tells us that the monoid of isometry classes is cancellative and therefore embeds into its group completion. Over commutative rings of higher dimension, this is no longer true and the present investigation concerns exactly this failure of cancellation for quadratic forms. While most of the contemporary work on quadratic forms is related to stable theories (hermitian  K-theory or higher Grothendieck--Witt groups), unstable questions related to the actual isometry classification seem to have mostly been neglected (except of course over rings of integers in local or global fields). 

The point of the paper is to show that homotopical methods can also be applied to  study the isometry classification of quadratic forms over schemes: if we restrict our attention to smooth affine schemes over infinite perfect fields of characteristic $\neq 2$, then the $\mathbb{A}^1$-representability result of \cite{gbundles2} allows to translate questions concerning classification of rationally hyperbolic quadratic forms into questions of obstruction-theoretic classification of morphisms into classifying spaces of orthogonal groups. Knowledge of $\mathbb{A}^1$-homotopy groups of the relevant classifying spaces can then be translated to classification results for rationally hyperbolic quadratic forms (and in particular stably trivial orthogonal bundles). 

Over smooth affine schemes of dimension $\leq 3$, where only the knowledge of $\mathbb{A}^1$-homotopy sheaves up to $\bm{\pi}_3^{\mathbb{A}^1}$ is required, we can actually give a complete description of the relevant invariants entering the classification of rationally hyperbolic forms. This is done via the classical sporadic isomorphisms for low-dimensional orthogonal resp. spin groups and the known $\mathbb{A}^1$-homotopy computations for the related groups. Surprisingly the identification of the orthogonal stabilization maps under the sporadic isomorphisms doesn't seem to be easy to find in the literature, necessitating a slightly extended discussion of the sporadic isomorphisms. The results can be used to produce many examples of stably trivial spin torsors over various varieties, cf. Sections~\ref{sec:dim3}. We also discuss the relation between classification of stably trivial spin torsors and quadratic bundles, which leads to a number of explicit examples of stably trivial quadratic bundles, cf. Section~\ref{sec:quad}. The most concise formulation of the combined results of the paper is the following: 

\begin{theorem}
Let $k$ be an infinite perfect field of characteristic $\neq 2$, let $X=\op{Spec}A$ be a smooth affine scheme of dimension $\leq 3$ over $k$ and let $(\mathscr{P},\phi)$ be a rationally hyperbolic quadratic form over $A$ of rank $n$ admitting a lift of structure group to $\op{Spin}(n)$. The only stable invariant of $(\mathscr{P},\phi)$ is the second Chern class. 
\end{theorem}

Here the second Chern class of a quadratic form  can be defined by stabilizing $(\mathscr{P},\phi)$ by adding hyperbolic planes to get a quadratic form of rank $\geq 7$ and then taking the second Chern class of the underlying projective bundle. 

While the above result identifies a single stable invariant, there are various other unstable invariants leading to various sources of examples of stably trivial quadratic bundles resp.  spin torsors. For the group $\op{SO}(3)$, stabilization induces multiplication by $2$ on the second Chern class which allows for examples related to $2$-torsion in Chow groups; for $\op{SO}(4)$ there are actually two Chern classes appearing in $\op{CH}^2$ leading to a great number of non-trivial stably trivial torsors. Since some of the low-dimensional spin groups are symplectic, there are examples of stably trivial spin torsors related to orientation information in $\widetilde{\op{CH}}^2$. Finally, there are various invariants arising from non-trivial $\bm{\pi}^{\mathbb{A}^1}_3$ which all become trivial after stabilization because $\bm{\pi}^{\mathbb{A}^1}_3{\op{B}}\op{Spin}(\infty)\cong\mathbf{K}^{\op{ind}}_3$ which is invisible in the torsor classification. These invariants give rise to interesting examples of stably trivial quadratic forms in a variety of situations. Since the examples arise via the sporadic isomorphisms, they can all be constructed very explicitly. 

Finally, it should be pointed out here that stably trivial quadratic forms refers here to Nisnevich-locally trivial torsors with structure group $\op{O}(n)$ which become trivial upon stabilization to $\op{O}(\infty)$. This is different from the usual notion of stably hyperbolic quadratic forms, which are quadratic forms which upon adding a hyperbolic form associated to a projective module become isomorphic to the hyperbolic form associated to a projective module. A theorem proved by Ojanguren and Pardon, cf. \cite[Section VIII.2]{knus}, states that over an integral affine scheme of dimension $\leq 3$, any rationally hyperbolic quadratic form is stably hyperbolic. In particular, all of the examples discussed in this paper eventually arise via the hyperbolic functor from vector bundles. What we provide here is the isometry classification of such things when the underlying scheme is smooth affine over an infinite field of characteristic unequal to $2$. The isometry classification actually allows to give examples of non-hyperbolic (stably hyperbolic by the above) forms of rank 4, cf. Example~\ref{ex:nonhyp}.

\subsection*{Conventions}
In this paper, $k$ always is an infinite perfect field of characteristic unequal to $2$. We consider smooth affine schemes $X=\op{Spec}A$ over $k$ and are interested in classification results for quadratic forms over the ring $A$. 

\section{Preliminaries on quadratic forms}
\label{sec:prelims}

This section provides a short recollection on the relevant facts concerning quadratic forms. Most of what is recalled below are standard definitions which can be found in any textbook, cf. e.g. \cite{knus}. 

\begin{definition}
Let $k$ be a field of characteristic unequal to $2$. 
\begin{itemize}
\item A \emph{quadratic form} over a commutative $k$-algebra $A$ is given by a finitely generated projective $A$-module $\mathscr{P}$ together with a map $\phi:\mathscr{P}\to A$ such that for each $a\in A$ and $x\in\mathscr{P}$ we have $\phi(ax)=a^2\phi(x)$ and $B_\phi(x,y)=\phi(x+y)-\phi(x)-\phi(y)$ is a symmetric bilinear form $B_\phi\in \op{Sym}^2(\mathscr{P}^\vee)$. 
\item 
The \emph{rank} of the quadratic form is defined to be the rank of the underlying projective module $\mathscr{P}$. 
\item A quadratic form $(\mathscr{P},\phi)$ is \emph{non-singular} if the morphism $\mathscr{P}\to\mathscr{P}^\vee:x\mapsto B_\phi(x,-)$ is an isomorphism.
\item An element $x\in\mathscr{P}$ is called \emph{isotropic} if $\phi(x)=0$. 
\item A morphism $f:(\mathscr{P}_1,\phi_1)\to(\mathscr{P}_2,\phi_2)$ of quadratic forms is an $A$-linear map $f:\mathscr{P}_1\to\mathscr{P}_2$ such that $\phi_2(f(x))=\phi_1(x)$ for all $x\in \mathscr{P}_1$. An isomorphism of quadratic forms is also called \emph{isometry}. The automorphism group of a quadratic form is called the \emph{orthogonal group} of the quadratic form.
\item Given two quadratic forms $(\mathscr{P}_1,\phi_1)$ and $(\mathscr{P}_2,\phi_2)$, there is a quadratic form
\[
(\mathscr{P}_1,\phi_1)\perp (\mathscr{P}_2,\phi_2):= (\mathscr{P}_1\oplus\mathscr{P}_2,\phi_1+\phi_2)
\]
called the \emph{orthogonal sum}.
\end{itemize}
\end{definition}

\begin{example}
Let $A$ be a commutative ring and $\mathscr{P}$ be a finitely generated projective module. Then there is a quadratic form whose underlying module is $\mathscr{P}\oplus\mathscr{P}^\vee$, equipped with the evaluation form $\op{ev}:(x,f)\mapsto f(x)$. The quadratic form $(\mathscr{P}\oplus\mathscr{P}^\vee,\op{ev})$ is called the \emph{hyperbolic space} associated to the projective module $\mathscr{P}$. In the special case where $\mathscr{P}=A$ is the free module of rank $1$, this is called the \emph{hyperbolic plane} $\mathbb{H}$ over $A$.
\end{example}

The following are the standard hyperbolicity notions from quadratic form theory, cf. \cite[Section VIII.2]{knus}.

\begin{definition}
A quadratic form is called \emph{hyperbolic}, if it is isometric to $\mathbb{H}(\mathscr{P})$ for some projective module $\mathscr{P}$.  A quadratic form $(\mathscr{P},\phi)$ is called \emph{stably hyperbolic} if there exists a projective module $\mathscr{Q}$ such that $(\mathscr{P},\phi)\perp\mathbb{H}(\mathscr{Q})$ is hyperbolic. A quadratic form $(\mathscr{P},\phi)$ over an integral domain $A$ is called \emph{rationally hyperbolic} if $(\mathscr{P},\phi)\otimes_A\op{Frac}(A)$ is hyperbolic.
\end{definition}

\begin{remark}
Note that stably hyperbolic forms are then necessarily of even rank, and stably hyperbolic forms are those that become $0$ in the Witt ring.
\end{remark}

For the purposes of the present paper, we will also be interested in stricter notions of stable triviality of quadratic forms:

\begin{definition}
A quadratic form $(\mathscr{P},\phi)$ is called \emph{stably trivial} if it becomes isometric to one of the split forms $\mathbb{H}^{\perp n}$ or $\mathbb{H}^{\perp n}\perp (A,a\mapsto a^2)$ after adding sufficiently many hyperbolic planes. 
\end{definition}

The stably trivial forms are those which represent classes in $\mathbb{Z}\cdot \mathbb{H}\oplus\mathbb{Z}\cdot [(A,a\mapsto a^2)]$ in the Grothendieck--Witt ring of $A$. 

The classical Witt cancellation theorem implies in particular that the notions of hyperbolic, stably trivial and stably hyperbolic all agree over fields: 

\begin{proposition}[Witt cancellation theorem]
Let $k$ be a field of characteristic $\neq 2$ and let $(V_1,\phi_1)$ and $(V_2,\phi_2)$ be two quadratic forms over $k$. If $(V_1,\phi_1)\oplus\mathbb{H}\cong (V_2,\phi_2)\oplus\mathbb{H}$ then $(V_1,\phi_1)\cong(V_2,\phi_2)$. In particular, a stably hyperbolic form is hyperbolic. 
\end{proposition}

\begin{remark}
\label{rem:abuse2}
All quadratic forms over fields considered in this paper are hyperbolic or of the form $\mathbb{H}^{\perp n}\oplus(A,a\mapsto a^2)$. In abuse of notation, the group $\op{SO}(n)$ will always denote the special orthogonal group associated to the \emph{split form} of rank $n$. I apologize to anyone who might be offended by this.
\end{remark}

Next, we will have a look at torsors under the various groups $\op{O}(n)$, $\op{SO}(n)$ and $\op{Spin}(n)$ and how they are related. In the end, we want to represent quadratic forms as suitable equivalence classes of spin torsors because the spin groups are easier to handle with the $\mathbb{A}^1$-homotopy methods.

First, we note that \'etale local triviality of quadratic forms implies that they can be viewed as torsors for the split orthogonal groups.

\begin{proposition}
\label{prop:formstorsors1}
Let $k$ be a field of characteristic unequal to $2$, let $X=\op{Spec}A$ be a $k$-scheme and denote by $\op{O}(n)$ the orthogonal group associated to the hyperbolic quadratic form on $k^n$. There is a functorial bijective correspondence between the set of isometry classes of quadratic forms over $A$ and the set $\op{H}^1_{\et}(X,\op{O}(n))$ of isomorphism classes of $\op{O}(n)$-torsors over $X$. 
\end{proposition}

\begin{proof}
A quadratic space is locally trivial in the \'etale topology by \cite[Corollary 1.2]{swan}. The remaining argument is standard. Given a quadratic form, we can find an \'etale cover $\bigsqcup_iU_i\to X$ over which the form trivializes and the transition morphisms are isometries. This implies that we get a morphism $\check{C}(U/X)\to{\op{B}}_{\et}\op{O}(n)$. Conversely, the cocycle condition implies that the locally trivial pieces can be glued to a quadratic form. Then one can check that simplicial homotopies correspond to globally defined isometries.
\end{proof}

There is a group extension $1\to\op{SO}(n)\to\op{O}(n)\to\mu_2\to 0$ which implies that the natural map ${\op{B}}_{\et}\op{SO}(n)\to{\op{B}}_{\et}\op{O}(n)$ is a degree 2 \'etale covering. Any quadratic form $(\mathscr{P},\phi)$ over $X$ induces a degree 2 \'etale covering $\tilde{X}\to X$, by pullback of the above along the classifying map $X\to{\op{B}}_{\et}\op{O}(n)$. This is the \emph{orientation covering} for the quadratic form $(\mathscr{P},\phi)$ whose class in $\op{H}^1_{\et}(X,\mu_2)$ is the first Stiefel--Whitney class  $\op{w}_1(\mathscr{P},\phi)$. We say that a quadratic form is \emph{orientable} if its orientation cover is the trivial degree 2 \'etale map $\op{id}\sqcup\op{id}:X\sqcup X\to X$. A choice of lift of $X\to{\op{B}}_{\et}\op{O}(n)$ to ${\op{B}}_{\et}\op{SO}(n)$ is called an \emph{orientation}. 

\begin{proposition}
\label{prop:formstorsors2}
Let $k$ be a field of characteristic unequal to $2$, let $X=\op{Spec}A$ be a $k$-scheme and denote by $\op{SO}(n)$ the special orthogonal group associated to the hyperbolic quadratic form on $k^n$. There is a functorial bijective correspondence between the set of isometry classes of orientable quadratic forms over $A$ and the set $\op{H}^1_{\et}(X,\op{SO}(n))$ of isomorphism classes of $\op{SO}(n)$-torsors over $X$. 
\end{proposition}

\begin{proof}
By what was said before, the orientability is the necessary and sufficient condition for lifting. The next claim is that there is effectively no choice for the lift. Consider the relevant homotopy sequence in the homotopy theory of simplicial sheaves on the big \'etale site
\[
[X,\op{O}(n)]\to[X,\mu_2]\to[X,{\op{B}}_{\et}\op{SO}(n)]\to[X,{\op{B}}_{\et}\op{O}(n)]
\]
whose exactness essentially states that the choice of lifts is the choice of orientations, up to isometries. The first map, induced by the projection $\op{O}(n)\to\mu_2$, is a surjection, saying that all orientations are equivalent up to isometry. Hence we get an injection of orientable forms into all forms.
\end{proof}

The functor $X\mapsto \op{H}^1_{\et}(X,\op{SO}(n))$ mapping a (smooth) scheme to the pointed set of isometry classes of orientable quadratic forms of rank $n$ is not representable in $\mathbb{A}^1$-homotopy, due to examples of non-trivial quadratic forms on affine spaces. However, as we recall later, the functor $X\mapsto \op{H}^1_{\op{Nis}}(X,\op{SO}(n))$ is $\mathbb{A}^1$-representable which is why we are interested in rationally hyperbolic forms here.

\begin{proposition}
\label{prop:shformstorsors}
Let $k$ be an infinite field of characteristic unequal to $2$, let $X=\op{Spec}A$ be a smooth affine $k$-scheme and denote by $\op{SO}(n)$ the special orthogonal group associated to the hyperbolic quadratic form on $k^n$.
The bijection of Proposition~\ref{prop:formstorsors2} restricts to a bijection from the isometry classes of rationally hyperbolic quadratic forms to the set 
$\op{H}^1_{\op{Nis}}(X,\op{SO}(n))$ of rationally trivial $\op{SO}(n)$-torsors. Moreover, this bijection restricts to an injection from stably trivial forms into rationally trivial $\op{SO}(n)$-torsors.
\end{proposition}

\begin{proof}
We can assume $X$ irreducible. Then the orientation cover of a rationally hyperbolic quadratic form is rationally trivial. But a finite \'etale map onto a smooth scheme which has a rational section is already trivial. Therefore, a rationally hyperbolic quadratic form over $X$ is orientable. Now the restriction of the bijection from Proposition~\ref{prop:formstorsors2} to rationally hyperbolic quadratic forms follows basically from Nisnevich's theorem identifying $\op{H}^1_{\op{Nis}}$ as torsors with a rational section. 

So we are left with proving the statement about stably trivial forms. We can again assume that $X$ is irreducible. By extension of scalars, a stably trivial form on $A$ gives rise to a stably hyperbolic form on the field of fractions $\op{Frac}(A)$. By Witt cancellation, the form on $\op{Frac}(A)$ is hyperbolic. By functoriality of the correspondence in Proposition~\ref{prop:formstorsors2}, the associated $\op{SO}(n)$-torsor over $X$ is rationally trivial.
\end{proof}

\section{Recollections on \texorpdfstring{$\mathbb{A}^1$}{A1}-homotopy theory}
\label{sec:a1homotopy}

In this section, we recall some basics of $\mathbb{A}^1$-homotopy theory, in particular the representability result and basic results of $\mathbb{A}^1$-obstruction theory.

We assume the reader is familiar with the basic definitions of $\mathbb{A}^1$-homotopy theory, cf.~\cite{MField}. Short introductions to those aspects relevant for the obstruction-theoretic torsor classification can be found in papers of Asok and Fasel, cf. e.g. \cite{AsokFasel,AsokFaselSplitting}. The notation in the paper generally follows the one from \cite{AsokFasel}. We generally assume that we are working over base fields of characteristic $\neq 2$. 

\subsection{Representability theorem}
\label{sec:representable}

The following $\mathbb{A}^1$-representability theorem, significantly generalizing an earlier result of Morel in \cite{MField}, has been proved in \cite{gbundles2}. 

\begin{theorem}
\label{thm:representability}
Let $k$ be an infinite field, and let $X=\op{Spec} A$ be a smooth affine $k$-scheme. Let $G$ be a reductive group such that each absolutely almost simple component of $G$ is isotropic. Then there is a bijection 
\[
\op{H}^1_{\op{Nis}}(X;G)\cong [X, {\op{B}}_{\op{Nis}}G]_{\mathbb{A}^1}
\]
between the pointed set of isomorphism classes of rationally trivial $G$-torsors over $X$ and the pointed set of $\mathbb{A}^1$-homotopy classes of maps $X\to {\op{B}}_{\op{Nis}}G$.
\end{theorem}

We discuss how this will be applied in the present work. Let $k$ be a field of characteristic $\neq 2$. In the abusive language of Remark~\ref{rem:abuse2}, $\op{SO}(n)$ is the special orthogonal group associated to the hyperbolic form of rank $n$. In particular, it is a semisimple absolutely almost simple group over $k$ which is isotropic (except for $n=4$ where it is an almost-product of two components with these properties). In particular, combining the representability theorem with Proposition~\ref{prop:shformstorsors} provides a bijection between the pointed set of rationally hyperbolic quadratic forms over $A$ and $[X,{\op{B}}_{\op{Nis}}\op{SO}(n)]_{\mathbb{A}^1}$. A similar statement applies to the groups $\op{Spin}(n)$: being associated to the hyperbolic forms, they are semisimple absolutely almost simple (again, with an exception in case $n=4$) and isotropic and this produces a bijection between rationally trivial $\op{Spin}(n)$-torsors and $[X,{\op{B}}_{\op{Nis}}\op{Spin}(n)]_{\mathbb{A}^1}$. 

Note also that the above classification result sets up a bijection between isomorphism classes of torsors and \emph{unpointed maps} to the classifying space $\op{B}_{\op{Nis}}G$. When we consider the spin groups which are $\mathbb{A}^1$-connected, the corresponding classifying spaces will be $\mathbb{A}^1$-simply connected which implies that there is a canonical bijection between pointed maps $X_+\to\op{B}_{\op{Nis}}\op{Spin}(n)$, the latter pointed by its canonical base point and unpointed maps $X\to\op{B}_{\op{Nis}}\op{Spin}(n)$. This is not true for the special orthogonal groups, where the sheaf of connected components is $\mathcal{H}^1_{\et}(\mu_2)$, the Nisnevich sheafification of the indicated \'etale cohomology presheaf. The classification of  unpointed maps is obtained by taking a quotient of the pointed maps by the action of the fundamental group sheaf. We will discuss this in Section~\ref{sec:quad} when we deduce statements concerning quadratic forms from the classification results for $\op{Spin}(n)$-torsors.

\subsection{Obstruction theory}

While the study of $\mathbb{A}^1$-homotopy classes of maps into classifying spaces may not seem an easier subject than the torsor classification, the other relevant tool actually allowing to prove some meaningful statements is obstruction theory. The basic statements concerning obstruction theory as applied to torsor classification can be found in various sources, such as \cite{MField} or \cite{AsokFasel,AsokFaselSplitting}. We only give a short list of the relevant statements which are enough for our purposes. 

Let $(\mathscr{Y},y)$ be a pointed $\mathbb{A}^1$-simply connected space. Then there is a sequence of pointed  $\mathbb{A}^1$-simply connected spaces, the Postnikov sections $(\tau_{\leq i}\mathscr{Y},y)$, with morphisms $p_i:\mathscr{Y}\to\tau_{\leq i}\mathscr{Y}$ and morphisms $f_i:\tau_{\leq i+1}\mathscr{Y}\to\tau_{\leq i}\mathscr{Y}$ such that
\begin{enumerate}
\item $\bm{\pi}_j^{\mathbb{A}^1}(\tau_{\leq i}\mathscr{Y})=0$ for $j>i$,
\item the morphism $p_i$ induces an isomorphism on $\mathbb{A}^1$-homotopy group sheaves in degrees $\leq i$,
\item the morphism $f_i$ is an $\mathbb{A}^1$-fibration, and the $\mathbb{A}^1$-homotopy fiber of $f_i$ is an Eilenberg--Mac Lane space of the form $\op{K}(\bm{\pi}_{i+1}^{\mathbb{A}^1}(\mathscr{Y}),i+1)$,
\item the induced morphism $\mathscr{Y}\to\tau_{\leq i}\op{holim}_i\mathscr{Y}$ is an $\mathbb{A}^1$-weak equivalence. 
\end{enumerate}
Moreover, $f_i$ is a principal $\mathbb{A}^1$-fibration, i.e., there is a morphism, unique up to $\mathbb{A}^1$-homotopy, 
$$
k_{i+1}:\tau_{\leq i}\mathscr{Y} \to\op{K}(\bm{\pi}_{i+1}^{\mathbb{A}^1}(\mathscr{Y}),i+2)
$$
called the $i+1$-th \emph{$k$-invariant} and an $\mathbb{A}^1$-fiber sequence
$$
\tau_{\leq i+1}\mathscr{Y} \to \tau_{\leq i}\mathscr{Y} \xrightarrow{k_{i+1}}\op{K}(\bm{\pi}_{i+1}^{\mathbb{A}^1}(\mathscr{Y}),i+2). 
$$ 
From these statements, one gets the following consequence: for a smooth $k$-scheme $X$ and a pointed $\mathbb{A}^1$-simply connected space  $\mathscr{Y}$, a given pointed map $g^{(i)}:X_+\to\tau_{\leq i}\mathscr{Y}$ lifts to a map $g^{(i+1)}:X_+\to\tau_{\leq i+1}\mathscr{Y}$ if and only if the following composite is null-homotopic:
$$
X_+\xrightarrow{g^{(i)}}\tau_{\leq i}\mathscr{Y}\to\op{K}(\bm{\pi}_{i+1}^{\mathbb{A}^1}(\mathscr{Y}),i+2), 
$$
or equivalently, if the corresponding obstruction class vanishes in the cohomology group $\op{H}^{i+2}_{\op{Nis}}(X;\bm{\pi}_{i+1}^{\mathbb{A}^1}(\mathscr{Y}))$. 

If this happens, then the possible lifts can be parametrized via the following exact sequence 
\[
[X_+,\Omega\tau_{\leq i}\mathscr{Y}]\to [X_+,\op{K}(\bm{\pi}_{i+1}^{\mathbb{A}^1}(\mathscr{Y}),i+1)]
\to [X_+,\Omega\tau_{\leq i+1}\mathscr{Y}] \to 
[X_+,\Omega\tau_{\leq i}\mathscr{Y}]
\]
where we can also explicitly identify $[X_+,\op{K}(\bm{\pi}_{i+1}^{\mathbb{A}^1}(\mathscr{Y}),i+1)]_{\mathbb{A}^1}\cong \op{H}^{i+1}_{\op{Nis}}(X;\bm{\pi}_{i+1}^{\mathbb{A}^1}(\mathscr{Y}))$. 

We want to state clearly what this means for the classification of spin torsors or quadratic forms over smooth affine schemes. If we have a  torsor for $G=\op{Spin}(n)$ or $G=\op{SO}(n)$, then the map into the respective classifying space associated by the representability theorem~\ref{thm:representability} is completely described by a sequence of classes in the lifting sets $\op{H}^{i+1}_{\op{Nis}}(X;\bm{\pi}_{i+1}^{\mathbb{A}^1}({\op{B}}_{\op{Nis}}G))$, which are well-defined only up to the respective action of $[X_+,\Omega\tau_{\leq i}{\op{B}}_{\op{Nis}}G]_{\mathbb{A}^1}$. Only indices $0\leq i+1\leq n$ can appear for schemes of dimension $n$ since the Nisnevich cohomological dimension equals the Krull dimension. Conversely, to construct a torsor, one needs a sequence of lifting classes as above, such that the associated obstruction classes in the groups $\op{H}^{i+2}_{\op{Nis}}(X;\bm{\pi}_{i+1}^{\mathbb{A}^1}({\op{B}}_{\op{Nis}}G))$ vanish. Put bluntly, $\mathbb{A}^1$-obstruction theory translates questions about $\mathbb{A}^1$-homotopy classes of maps (from smooth schemes) into computations of certain (finitely many) cohomology classes. 
As a result, the classification of $\op{Spin}(n)$-torsors over smooth affine schemes of dimension $\leq 3$ requires only knowledge of the first three $\mathbb{A}^1$-homotopy sheaves of ${\op{B}}_{\op{Nis}}\op{Spin}(n)$. This information can be recovered from known computations of $\mathbb{A}^1$-homotopy sheaves for special linear and symplectic groups via the sporadic isomorphisms.

\section{Recollections on sporadic isomorphisms}
\label{sec:sporadic}

In this section, we provide some information on the sporadic isomorphisms identifying the low-rank spin groups with other low-rank groups (for which the relevant low-dimensional $\mathbb{A}^1$-homotopy sheaves have already been computed). Since we are interested in stabilization results and the classification of stably hyperbolic forms, we want to obtain more precisely that the sequence of stabilization morphisms  for the spin groups from $\op{Spin}(3)$ to $\op{Spin}(6)$ corresponds, under the sporadic isomorphisms, to the sequence 
\[
\op{SL}_2\xrightarrow{\Delta}\op{SL}_2\times\op{SL}_2 \xrightarrow{(2\alpha,2\beta)} \op{Sp}_4\xrightarrow{\iota}\op{SL}_4
\]
where $\Delta$ is the diagonal embedding, $(2\alpha,2\beta)$ is the embedding arising from the long roots for $\op{Sp}_4$, and $\iota$ is the natural embedding of $\op{Sp}_4$ as stabilizer of the standard symplectic form. This can be done by realizing the usual models of the sporadic isomorphisms inside the one for $\op{SO}(6)$. With this goal in mind, parts of the development will differ slightly from the common presentation of sporadic isomorphisms which doesn't pay respect to the stabilization morphisms. Still, most of the following will be well-known and familiar to many, cf. e.g. Garrett's notes \cite{garrett}.

We begin by recalling the identification of $\op{SL}_4$ with $\op{Spin}(6)$. Consider the $4$-dimensional $k$-vector space $V=k^4$ with the natural action of $\op{SL}_4$. This induces a natural action of $\op{SL}_4(k)$ on the $6$-dimensional space $V^{\wedge 2}$, i.e., a representation $\op{SL}_4\to\op{SL}_6$. On $V^{\wedge 2}$ there is a natural symmetric bilinear form
\[
\langle-,-\rangle:V^{\wedge 2}\times V^{\wedge 2}\to k:\langle v_1\wedge w_1,v_2\wedge w_2\rangle=\det(v_1,w_1,v_2,w_2).
\]
The form is non-degenerate and hyperbolic with an orthogonal basis given by 
\[
(\op{e}_1\wedge \op{e}_2)\pm (\op{e}_3\wedge \op{e}_4), \quad (\op{e}_1\wedge \op{e}_3)\pm (\op{e}_2\wedge \op{e}_4), \quad (\op{e}_1\wedge \op{e}_4)\pm (\op{e}_2\wedge \op{e}_3).
\]
The induced action of $\op{SL}_4$ on $V^{\wedge 2}$ will preserve this form, giving a homomorphism $\op{SL}_4\to\op{SO}(6)$. It can be checked via the Lie algebra that the kernel is finite, equal to the $\{\pm 1\}$, hence the homomorphism $\op{SL}_4\to\op{SO}(6)$ induces the sporadic isomorphism $\op{SL}_4\cong\op{Spin}(6)$. This implies the following:

\begin{proposition}
\label{prop:so6}
The morphism ${\op{B}}\op{SL}_4\to{\op{B}}_{\op{Nis}}\op{SO}(6)$ induced by the sporadic isogeny $\op{SL}_4\to\op{SO}(6)$ is given as follows: if $R$ is a commutative ring and $\mathscr{P}$ is an oriented projective $R$-module of rank $4$, then the associated quadratic form of rank $6$ is given by the projective $R$-module $\mathscr{P}^{\wedge 2}$ equipped with the evaluation form
\[
\mathscr{P}^{\wedge 2}\otimes\mathscr{P}^{\wedge 2}\to \mathscr{P}^{\wedge 4}\cong R, 
\]
where the first map is the projection from the tensor product to the exterior product and the second isomorphism is the orientation of $\mathscr{P}$.
\end{proposition}

Next, we consider the sporadic isomorphism $\op{Sp}_4\cong\op{Spin}(5)$ and its relation to the description of $\op{Spin}(6)$ obtained above. There is a natural embedding of $\op{Sp}_4$ into $\op{SL}_4$ as subgroup of matrices preserving a symplectic form on $V=k^{4}$. The composition with the above identification provides a group homomorphism $\op{Sp}_4\to\op{SL}_4\to\op{SO}(6)$ arising from the induced action of $\op{Sp}_4$ on $V^{\wedge 2}$. Viewing the symplectic form on $V$ as a linear form $\omega:V^{\wedge 2}\to k$ gives a decomposition of the quadratic space $V$ as direct sum of the $5$-dimensional quadratic space $W=\ker\omega$ with a  line equipped with the standard form $x\mapsto x^2$. Now the action of $\op{Sp}_4$ on $V^{\wedge 2}$ will preserve $W=\ker \omega$, giving us a morphism $\op{Sp}_4\to\op{SO}(5)$. Again, it can be checked using the Lie algebra that this induces an isomorphism $\op{Sp}_4\cong\op{Spin}(5)$. 

We have therefore proved the following:

\begin{proposition}
\label{prop:so5}
The morphism ${\op{B}}\op{Sp}_4\to{\op{B}}_{\op{Nis}}\op{SO}(5)$ induced by the sporadic isogeny $\op{Sp}_4\to\op{SO}(5)$ is given as follows: let $R$ be a commutative ring and let $\mathscr{P}$ be a symplectic module of rank $2$, i.e., a projective module $\mathscr{P}$ of rank $4$ equipped with a symplectic form $\omega:\mathscr{P}^{\wedge 2}\to R$. The corresponding quadratic form of rank $5$ is given by the projective module $\ker\omega$ equipped with the evaluation form
\[
\ker\omega\hookrightarrow \mathscr{P}^{\wedge 2}\otimes\mathscr{P}^{\wedge 2}\to \mathscr{P}^{\wedge 4}\cong R.
\]
\end{proposition}

Moreover, the decomposition of the six-dimensional quadratic space $V$ as direct sum of $W$ and a line implies that we can in fact identify the stabilization morphism.

\begin{proposition}
\label{prop:stab56}
There is a commutative diagram
\[
\xymatrix{
\op{Sp}_4 \ar[r] \ar[d]_\cong & \op{SL}_4\ar[d]^\cong \\
\op{Spin}(5) \ar[r] & \op{Spin}(6)
}
\]
where the top horizontal is the natural embedding, the bottom horizontal is the stabilization morphism, and the verticals are the sporadic isomorphisms. 
\end{proposition}


Now we will deal with the sporadic isomorphism $\op{Spin}(4)\cong\op{SL}_2\times\op{SL}_2$ and its relation with the isomorphisms discussed previously. If we write the $4$-dimensional space $V$ with the symplectic form $\omega$ as a direct sum of two $2$-dimensional symplectic spaces, the sporadic isomorphism $\op{SL}_2\cong\op{Sp}_2$ induces natural embeddings $\op{SL}_2\times\op{SL}_2\hookrightarrow\op{Sp}_4\hookrightarrow \op{SL}_4$.
The first embedding is the one given by the long roots in $\op{Sp}_4$. The composite is the embedding of a Levi subgroup of the parabolic subgroup of $\op{SL}_4$ preserving the first of the two-dimensional subspaces.

We first set up the sporadic isomorphism $\op{SL}_2\times \op{SL}_2\cong \op{Spin}(4)$ and then show how this identification fits with the stabilization to $\op{Spin}(5)$. The following is the split version of the classical identification of $\op{SL}_2\times\op{SL}_2$ via its action on the quaternions. 

\begin{proposition}
\label{prop:so4}
Consider the matrix algebra $\op{Mat}_{2\times 2}(k)$ equipped with the action of $\op{SL}_2\times \op{SL}_2$ given by 
\[
\left(A=\left(\begin{array}{cc}
a_{11}&a_{12}\\a_{21}&a_{22}\end{array}\right), 
B=\left(\begin{array}{cc}
b_{11}&b_{12}\\b_{21}&b_{22}\end{array}\right),M\right)\mapsto  A\cdot M\cdot B^{-1}.
\]
On the matrix algebra, there is a non-degenerate symmetric bilinear form, the modified trace form $\langle X,Y\rangle =-\op{tr}(X\cdot WY^{\op{t}}W^{-1})$ where 
\[
W=\left(\begin{array}{cc}0&-1\\1&0\end{array}\right).
\]
The corresponding quadratic form is $2\det$; it is hyperbolic and preserved by the action of $\op{SL}_2\times\op{SL}_2$. 

Therefore, the above action of $\op{SL}_2\times\op{SL}_2$ on the matrix algebra induces an isomorphism $\op{SL}_2\times\op{SL}_2\cong\op{Spin}(4)$. The corresponding morphism  ${\op{B}}\op{SL}_2\times{\op{B}}\op{SL}_2\to{\op{B}}_{\op{Nis}}\op{SO}(4)$ of classifying spaces maps an $\op{SL}_2\times\op{SL}_2$-torsor to the associated bundle for the above representation.
\end{proposition}

\begin{proposition}
\label{prop:stab45}
Consider the action of $\op{SL}_2\times\op{SL}_2$ on $V^{\wedge 2}$ via the composition 
\[
\op{SL}_2\times\op{SL}_2\hookrightarrow\op{SL}_4\to\op{SO}(6).
\]
The action is trivial on  the subspace $\langle(\op{e}_1\wedge\op{e}_2)\pm(\op{e}_3\wedge\op{e}_4)\rangle$. Equipped with the restriction of the determinant form from $V^{\wedge 2}$, it is a hyperbolic plane.

The map $\op{Mat}_{2\times 2}(k)\to V^{\wedge 2}$ given by
\begin{eqnarray*}
\left(\begin{array}{cc}1&0\\0&0\end{array}\right)\mapsto \op{e}_4\wedge \op{e}_1, &&
\left(\begin{array}{cc}0&1\\0&0\end{array}\right)\mapsto \op{e}_1\wedge \op{e}_3, \\
\left(\begin{array}{cc}0&0\\1&0\end{array}\right)\mapsto \op{e}_4\wedge \op{e}_2, &&
\left(\begin{array}{cc}0&0\\0&1\end{array}\right)\mapsto \op{e}_2\wedge\op{e}_3
\end{eqnarray*}
is a morphism of $\op{SL}_2\times\op{SL}_2$-representations and of quadratic spaces which induces an isomorphism and isometry onto its image. 

With the identifications from Proposition~\ref{prop:so4} and \ref{prop:so5}, the morphism 
\[
\op{SL}_2\times\op{SL}_2\cong\op{Spin}(4)\hookrightarrow \op{Spin}(5)\cong \op{Sp}_4
\]
induced by the stabilization morphism is the long-root embedding. 
\end{proposition}

\begin{proof}
I just spent the afternoon doing the flipping computations. 
\end{proof}

\begin{remark}
Of course the classical branching rules tell us that the restriction of the six-dimensional representation of $\op{SL}_4$ to $\op{SL}_2\times\op{SL}_2$ is the direct sum of the natural $4$-dimensional representation (corresponding to the identification with $\op{Spin}(4)$) and a $2$-dimensional trivial representation. But we need to identify exactly the morphisms on the groups to compute the induced maps on homotopy.
\end{remark}

Finally, we get to the sporadic isomorphism for the smallest group. 

\begin{proposition}
\label{prop:so3}
Consider the diagonal embedding $\Delta:\op{SL}_2\to\op{SL}_2\times\op{SL}_2$. Then $\op{SL}_2$ acts on $\op{Mat}_{2\times 2}(k)$ by conjugation. The action is trivial on the subspace spanned by the identity matrix. The action preserves the matrices of trace $0$. The restriction of the modified trace form to the subspace of trace $0$ matrices coincides with the trace form $\langle X,Y\rangle=\op{tr}(X\cdot Y)$, which is preserved by the action of $\op{SL}_2$. This induces an isomorphism $\op{SL}_2\cong\op{Spin}(3)$. The induced map ${\op{B}}\op{SL}_2\to{\op{B}}_{\op{Nis}}\op{SO}(3)$ on classifying spaces maps an $\op{SL}_2$-torsor to the associated vector bundle for this representation.
\end{proposition}

\begin{proposition}
\label{prop:stab34}
With the identifications of Propositions~\ref{prop:so4} and \ref{prop:so3}, the map 
\[\op{SL}_2\cong\op{Spin}(3)\to\op{Spin}(4)\cong\op{SL}_2\times\op{SL}_2
\]
induced by stabilization of the quadratic form is the diagonal embedding $\Delta$. 
\end{proposition}

A direct computation of the morphism $\op{SL}_4\to\op{SO}(6)$ shows that the composition $\op{SL}_3\to\op{SL}_4\to\op{SO}(6)$ induces the hyperbolic morphism ${\op{B}}\op{SL}_3\to{\op{B}_{\op{Nis}}}\op{SO}(6)$, cf. e.g. the proof of \cite[Proposition 2.3.1]{octonion}.

\begin{proposition}
\label{prop:hypso4}
With the identification of Proposition~\ref{prop:so4}, the composition 
\[
\op{SL}_2\xrightarrow{\iota_2}\op{SL}_2\times\op{SL}_2\to \op{SO}(4)
\]
induces the hyperbolic morphism ${\op{B}}\op{SL}_2\to{\op{B}}\op{SO}(4)$. 
\end{proposition}

\begin{proof}
A direct computation of the action of $\op{SL}_2$ on the 4-dimensional matrix space shows that it is conjugate to the map $\op{SL}_2\to\op{SL}_4$ which sends a matrix $M$ to the block matrix whose two blocks are $M$ and $(M^{-1})^{\op{t}}$. This proves the claim.
\end{proof}

\section{Sporadic results on spin torsors}
\label{sec:dim3}

In this section, we discuss the classification of stably trivial spin torsors over smooth affine schemes of dimension $\leq 3$. The results will be based on the discussion of the sporadic isomorphisms in Section~\ref{sec:sporadic}. We will explain in Section~\ref{sec:quad} how the results from the present section translate to the quadratic form classification. 

Note that the sporadic isomorphisms imply that all the spin groups up to $\op{Spin}(6)$ are special in the sense of Serre. Therefore, the results below will in fact provide a classification of all $\op{Spin}(n)$-torsors for $n\leq 6$ on smooth affine schemes of dimension $\leq 3$. Since there is no difference between Nisnevich- and \'etale-local triviality of torsors, we omit the indices in $\op{B}_{\et}=\op{B}_{\op{Nis}}$. The various invariants  relevant for the classification of the spin bundles will sit in degrees 2 and 3; but the only stable invariant for dimension $\leq 3$ will be the second Chern class in $\op{CH}^2(X)$.

The results below exhibit essentially three different types of examples of stably trivial spin torsors on smooth affine schemes of dimension $\leq 3$. One type of examples comes from the changes in $\bm{\pi}_2^{\mathbb{A}^1}$ of the classifying spaces of spin groups in low ranks, where for $\op{Spin}(3)$ and $\op{Spin}(5)$ we have $\mathbf{K}^{\op{MW}}_2$ and consequently the lifting classes live in $\widetilde{\op{CH}}^2(X)$ which has some additional quadratic information not present in $\op{CH}^2(X)$. Moreover, $\bm{\pi}^{\mathbb{A}^1}_2({\op{B}}\op{Spin}(4))\cong\mathbf{K}^{\op{MW}}_2\times\mathbf{K}^{\op{MW}}_2$ which means that there are quite a lot stably trivial spin torsors of rank 4. The second type of example for low ranks can be traced to $\bm{\pi}_3^{\mathbb{A}^1}{\op{B}}\op{SL}_2$ which provides various types of stably trivial spin torsors related to stably free modules; these will already be trivial by stabilization to $\op{Spin}(6)$. Finally, the last type of examples are $\op{Spin}(6)$-torsors detected by lifting classes in $\op{CH}^3(X)$ which become trivial by stabilization to $\op{Spin}(7)$. 

We proceed from higher ranks to lower ranks, analysing every time the classification of all torsors and which torsors become trivial upon passing to higher ranks. 

\subsection{Remark on the stable range}

The first statement to make is that the relevant homotopy groups of classifying spaces of spin groups are stable from $\op{Spin}(7)$ on, i.e., the natural maps ${\op{B}}_{\op{Nis}}\op{Spin}(n)\to{\op{B}}_{\op{Nis}}\op{Spin}(n+1)$ induce isomorphisms on $\bm{\pi}_i^{\mathbb{A}^1}$ for $i\leq 3$ and $n\geq 7$. Part of this stabilization result was already established in \cite[Theorem 6.8]{torsors}. The other half of the stabilization results can be proved as in \cite{torsors} using the $\mathbb{A}^1$-fiber sequence $\op{Q}_{2n}\to{\op{B}}_{\op{Nis}}\op{Spin}(2n)\to{\op{B}}_{\op{Nis}}\op{Spin}(2n+1)$ and the identification of $\op{Q}_{2n}$ as motivic sphere from \cite{AsokDoranFasel}. This brings down the stable range to $n\geq 8$. Getting it down to $n\geq 7$ uses the octonion multiplication, cf. \cite[Corollary 3.4.3]{octonion}. 

With this stabilization at hand, it is then clear that there are no non-trivial stably trivial spin torsors for $\op{Spin}(n)$, $n\geq 7$ over smooth affine schemes of dimension $\leq 3$. The low-dimensional $\mathbb{A}^1$-homotopy sheaves for the spin groups $\op{Spin}(n)$ with $n\geq 7$ are given explicitly as follows, cf. \cite[Section 3.4]{octonion}: 
\[
\bm{\pi}^{\mathbb{A}^1}_1{\op{B}}_{\op{Nis}}\op{Spin}(n)=0,\; \bm{\pi}^{\mathbb{A}^1}_2{\op{B}}_{\op{Nis}}\op{Spin}(n)=\mathbf{K}^{\op{M}}_2,\; \bm{\pi}^{\mathbb{A}^1}_3{\op{B}}_{\op{Nis}}\op{Spin}(n)=\mathbf{K}^{\op{ind}}_3.
\]
Since $\op{H}^3_{\op{Nis}}(X,\mathbf{K}^{\op{ind}}_3)=0$, cf.~\cite[Lemma 3.2.1]{octonion}, there is only one interesting lifting class for the torsor classification which is the second Chern class in $\op{H}^2_{\op{Nis}}(X,\mathbf{K}^{\op{M}}_2)\cong\op{CH}^2(X)$. In particular, for a smooth affine scheme $X$ over an infinite field of characteristic unequal to $2$, of dimension $\leq 3$, the second Chern class induces a bijection 
\[
\op{c}_2:\op{H}^1_{\op{Nis}}(X,\op{Spin}(7))\xrightarrow{\cong} \op{CH}^2(X).
\]

\subsection{Stably trivial torsors of rank $6$}
We start by analysing stably trivial $\op{Spin}(6)$-torsors. By the sporadic isomorphism these are classified in the same way as rank $4$ vector bundles, and over schemes of dimension $3$ the latter are all determined by their Chern classes.

\begin{proposition}
\label{prop:classifyso6}
Let $k$ be an infinite perfect field of characteristic unequal to $2$, let $X=\op{Spec} A$ be a smooth affine scheme over $k$ of dimension $\leq 3$. We have the following statements for classification of maps $X\to{\op{B}}\op{Spin}(6)$:
\begin{enumerate}
\item There is an isomorphism $\bm{\pi}^{\mathbb{A}^1}_2{\op{B}}\op{Spin}(6)\cong\mathbf{K}^{\op{M}}_2$. The first non-trivial lifting class lives in $\op{H}^2_{\op{Nis}}(X,\mathbf{K}^{\op{M}}_2)\cong\op{CH}^2(X)$. In particular, if $X$ has dimension $\leq 2$, the second Chern class provides a bijection 
\[
\op{c}_2:[X,{\op{B}}_{\op{Nis}}\op{Spin}(6)]\cong\op{CH}^2(X). 
\]
\item There is an isomorphism $\bm{\pi}^{\mathbb{A}^1}_3{\op{B}}\op{Spin}(6)\cong\mathbf{K}^{\op{Q}}_3$. Using the identification $\op{H}^3_{\op{Nis}}(X,\mathbf{K}^{\op{Q}}_3)\cong\op{CH}^3(X)$,
there is an exact sequence
\[
\op{H}^1(X,\mathbf{K}^{\op{M}}_2)\xrightarrow{\Omega k_3}\op{CH}^3(X)\to [X,{\op{B}}_{\op{Nis}}\op{Spin}(6)]\xrightarrow{\op{c}_2}\op{CH}^2(X)\to 0. 
\]
Here the map $\Omega k_3$ is the looping of the Postnikov invariant. 

Over an algebraically closed field, the sequence splits and the invariants of a $\op{Spin}(6)$-torsor of the form $\mathscr{P}^{\wedge 2}$ are given by the Chern classes of the oriented projective rank $4$ module $\mathscr{P}$.
\end{enumerate}
\end{proposition}

\begin{proof}
  We use the identification of Proposition~\ref{prop:so6}, whence it suffices to analyse the obstruction theory for maps into ${\op{B}}\op{SL}_4$. By \cite{torsors}, the first three $\mathbb{A}^1$-homotopy sheaves are stable and equal to the corresponding Quillen K-theory sheaves. The realizability of all lifting classes follows since the obstruction classes would live in degrees above the Nisnevich cohomological dimension of $X$. The statements made are then direct applications of the obstruction-theoretic formalism, cf.~Section~\ref{sec:a1homotopy}. 
For smooth affine $3$-folds over algebraically closed, it is known, cf. \cite{KumarMurthy} or \cite{AsokFasel} (Theorem 6.11 in v1 on the arXiv, unfortunately removed from the published paper), that the set of oriented vector bundles is actually identified with $\op{CH}^2(X)\times \op{CH}^3(X)$ via the Chern classes. Over non-algebraically closed fields, there could be some problems with torsion classes annihilated by the order of the Postnikov invariant $k_3$. 
\end{proof}

We need to analyse the stabilization from $\op{Spin}(6)$ to $\op{Spin}(n)$, $n\geq 7$. To deal with the third $\mathbb{A}^1$-homotopy sheaf, recall the following computation from  \cite[Section 3.4]{octonion}. 

\begin{proposition}
\label{prop:pi3so6}
The induced morphism
\[
\mathbf{K}^{\op{Q}}_3\cong\bm{\pi}_3^{\mathbb{A}^1}{\op{B}}\op{SL}_4\to \bm{\pi}_3^{\mathbb{A}^1}{\op{B}_{\op{Nis}}}\op{Spin}(6)\to \bm{\pi}_3^{\mathbb{A}^1}{\op{B}_{\op{Nis}}}\op{Spin}(7)\cong\mathbf{K}^{\op{ind}}_3
\]
is the natural projection. 
\end{proposition}

\begin{corollary}
Let $k$ be an infinite perfect field of characteristic unequal to $2$, let $X=\op{Spec} A$ be a smooth affine scheme of dimension $\leq 3$ over $k$. A $\op{Spin}(6)$-torsor is stably trivial if and only if its second Chern class is trivial. In particular, the stably trivial spin torsors of rank $6$ are in fact classified by 
\[
\op{coker}\left(\Omega k_3:\op{H}^1(X,\mathbf{K}^{\op{M}}_2)\to\op{CH}^3(X)\right). 
\]
\end{corollary}

\begin{proof}
The morphism ${\op{B}}\op{Spin}(6)\to{\op{B}}_{\op{Nis}}\op{Spin}(\infty)$ induces an isomorphism on $\bm{\pi}^{\mathbb{A}^1}_2$, by $\mathbb{A}^1$-$2$-connectedness of $\op{Q}_6\cong\op{hofib}\left({\op{B}}_{\op{Nis}}\op{Spin}(6)\to{\op{B}}_{\op{Nis}}\op{Spin}(7)\right)$. The stable value of the second homotopy group is $\bm{\pi}^{\mathbb{A}^1}_2({\op{B}}_{\op{Nis}}\op{Spin}(n))\cong \mathbf{K}^{\op{M}}_2$ for $n\geq 6$. In particular, the lifting class in $\op{H}^2_{\op{Nis}}(X,\mathbf{K}^{\op{M}}_2)\cong\op{CH}^2(X)$, which is the second Chern class, is a stable invariant. 

The projection in Proposition~\ref{prop:pi3so6} induces the zero map 
\[
\op{CH}^3(X)\cong\op{H}^3_{\op{Nis}}(X,\mathbf{K}^{\op{Q}}_3)\to\op{H}^3_{\op{Nis}}(X,\mathbf{K}^{\op{ind}}_3)=0. 
\]
This implies that the third Chern class of the rank $6$ quadratic form is  an unstable invariant, and there is no stable invariant of degree $3$ for quadratic forms since $\op{H}^3_{\op{Nis}}(X,\mathbf{K}^{\op{ind}}_3)=0$. 

Combining the above assertions shows that a $\op{Spin}(6)$-torsor is stably trivial if and only if its second Chern class is trivial. The cokernel claim follows from Proposition~\ref{prop:classifyso6}.
\end{proof}

This provides many examples of stably trivial spin torsors over affine 3-folds, compare to a similar class of examples in \cite[Example 4.2.3]{octonion}.

\begin{example}
\label{ex:so6}
Let $\overline{X}$ be a smooth projective variety of dimension $3$ over $\mathbb{C}$ such that $\op{H}^0(\overline{X},\omega_{\overline{X}})\neq 0$, i.e., there is a global non-trivial holomorphic $3$-form. Let $X$ be a the complement of a divisor in $\overline{X}$. By \cite[Theorem 2]{MurthySwan} and \cite[Proposition 2.1 and Corollary 5.3]{BlochMurthySzpiro}, the Chow group $\op{CH}^3(X)$ is a divisible torsion-free group of uncountable rank. In particular, there are uncountably many isomorphism classes of $\op{Spin}(6)$-torsors which are stably trivial.
\end{example}

\subsection{Stably trivial torsors in rank $5$}
The next step is now to analyse the classification of torsors of rank $5$ and check which of these become trivial by passage to $\op{Spin}(6)$ (or by adding a hyperbolic plane). By the sporadic isomorphism, the relevant information is contained in the symplectic group $\op{Sp}_4$. 

\begin{proposition}
\label{prop:classifyso5}
Let $k$ be an infinite perfect field of characteristic unequal to $2$, let $X=\op{Spec} A$ be a smooth affine scheme over $k$ of dimension $\leq 3$. We have the following statements for classification of maps $X\to{\op{B}}\op{Spin}(5)$:
\begin{enumerate}
\item There is an isomorphism $\bm{\pi}^{\mathbb{A}^1}_2{\op{B}}\op{Spin}(5)\cong\mathbf{K}^{\op{MW}}_2$. The first non-trivial lifting class lives in $\op{H}^2_{\op{Nis}}(X,\mathbf{K}^{\op{MW}}_2)\cong\widetilde{\op{CH}}^2(X)$. In particular, if $X$ has dimension $\leq 2$, then the first Pontryagin class (as an invariant in the Chow--Witt ring of ${\op{B}}\op{Sp}_4$) provides a bijection 
\[
\op{p}_1:[X,{\op{B}}\op{Spin}(5)]\cong\widetilde{\op{CH}}^2(X). 
\]
\item There is an isomorphism $\bm{\pi}^{\mathbb{A}^1}_3{\op{B}}\op{Spin}(5)\cong\mathbf{KSp}_3$, and an exact sequence
\[
\op{H}^1_{\op{Nis}}(X,\mathbf{K}^{\op{MW}}_2)\xrightarrow{\Omega k_3} \op{H}^3_{\op{Nis}}(X,\mathbf{KSp}_3)\to [X,{\op{B}}\op{Spin}(5)]\xrightarrow{\op{p}_1}\widetilde{\op{CH}}^2(X)\to 0.
\]
The invariants are the characteristic classes of the symplectic bundle corresponding to the $\op{Spin}(5)$-torsor.
\end{enumerate}
\end{proposition}

\begin{proof}
We use the identification of Proposition~\ref{prop:so5}. Then it suffices to analyse the obstruction theory for maps into ${\op{B}}\op{Sp}_4$. By \cite{torsors}, the first four homotopy sheaves of $\op{Sp}_4$ are stable and equal the respective symplectic K-groups. This implies the claim on homotopy groups. As in Proposition~\ref{prop:classifyso6} all classes are realizable because the relevant obstructions live above the Nisnevich cohomological dimension of $X$. The remaining claims are explicit formulations of the obstruction-theoretic statements in Section~\ref{sec:a1homotopy}. 
\end{proof}

Recall from \cite[Proposition 4.16]{AsokFasel} that we have the following presentation 
\[
\op{Ch}^2(X)\stackrel{\op{Sq}^2}{\longrightarrow}\op{Ch}^3(X)\to \op{H}^3_{\op{Nis}}(X,\mathbf{KSp}_3)\to 0.
\]
The corresponding invariant can be non-trivial, and classically is the invariant $\alpha$ used by Atiyah and Rees to describe complex plane bundles over $\op{S}^6$. However, for $X$ a smooth affine $3$-fold over an algebraically closed field, we have $\op{Ch}^3(X)=0$ because of the unique divisibility of top Chow groups. In particular, we get the following:

\begin{corollary}
Let $k$ be an algebraically closed field and let $X$ be a smooth affine scheme of dimension $\leq 3$ over $k$. Then the first Pontryagin class induces a bijection 
\[
\op{p}_1:\op{H}^1_{\op{Nis}}(X,\op{Spin}(5))\xrightarrow{\cong}\widetilde{\op{CH}}^2(X).
\]
\end{corollary}

Now we discuss the behaviour of the lifting classes under the stabilization morphism.

\begin{proposition}
\begin{enumerate}
\item 
The  morphism
\[
\mathbf{K}^{\op{MW}}_2\cong \bm{\pi}_2^{\mathbb{A}^1}{\op{B}}\op{Spin}(5)\to 
\bm{\pi}_2^{\mathbb{A}^1}{\op{B}}\op{Spin}(6)\cong\mathbf{K}^{\op{M}}_2
\]
induced by the stabilization homomorphism $\op{Spin}(5)\to\op{Spin}(6)$ is the natural projection. 
\item
The morphism
\[
\mathbf{KSp}_3\cong \bm{\pi}_3^{\mathbb{A}^1}{\op{B}}\op{Spin}(5)\to 
\bm{\pi}_3^{\mathbb{A}^1}{\op{B}}\op{Spin}(6)\cong\mathbf{K}^{\op{Q}}_3
\]
induced by the stabilization homomorphism $\op{Spin}(5)\to\op{Spin}(6)$ is the forgetful morphism.
\end{enumerate}
\end{proposition}

\begin{proof}
This follows directly from Proposition~\ref{prop:stab56} and the fact that the forgetful morphism (from symplectic bundles to oriented vector bundles) is compatible with stabilization. 
\end{proof}

\begin{lemma}
Let $k$ be an infinite perfect field of characteristic $\neq 2$ and let $X$ be a smooth scheme. Then the morphism $\op{H}^3_{\op{Nis}}(X,\mathbf{KSp}_3)\to \op{H}^3_{\op{Nis}}(X,\mathbf{K}^{\op{Q}}_3)$ induced by the forgetful morphism is the zero map.
\end{lemma}

\begin{proof}
Note that any morphism $\mathbf{KSp}_3\to\mathbf{K}^{\op{Q}}_3$ will induce a morphism of Gersten complexes, and we can use that to compute the induced morphism on cohomology. A cycle representing a class in $\op{H}^3_{\op{Nis}}(X,\mathbf{KSp}_3)$ will be a finite sum, indexed by codimension $3$ points $x$ of $X$, of elements in $\op{GW}^3_0(k(x))\cong\mathbb{Z}/2\mathbb{Z}$, cf. \cite[Section 4]{AsokFasel}. On the other hand, the degree $3$ cycle group in the Gersten complex for $\mathbf{K}^{\op{Q}}_3$ will be a direct sum of copies of $\mathbb{Z}$ indexed by the codimension $3$ points. The induced map $\op{GW}^3_0(k(x))\to\mathbb{Z}$ must necessarily be the zero map. This shows that any morphism $\mathbf{KSp}_3\to\mathbf{K}^{\op{Q}}_3$ will induce the zero map in degree $3$ Nisnevich cohomology.
\end{proof}

\begin{corollary}
\label{cor:stabso5}
Let $k$ be an infinite perfect field of characteristic unequal to $2$, let $X=\op{Spec}A$ be a smooth affine scheme of dimension $\leq 3$ over $k$. A spin torsor of rank $5$ over $A$ is stably trivial if and only if its second Chern class in $\op{CH}^2(X)$ is trivial. 

In particular, we have two invariants detecting stably trivial spin torsors of rank $5$. The first invariant lives in  $\ker\left(\widetilde{\op{CH}}^2(X)\to\op{CH}^2(X)\right)$. 
If the first invariant vanishes, then there is a secondary invariant which lives in 
\[
\op{coker}\left(\Omega k_3:\op{H}^1_{\op{Nis}}(X,\mathbf{K}^{\op{MW}}_2)\to \op{H}^3_{\op{Nis}}(X,\mathbf{KSp}_3)\right). 
\]
\end{corollary}

A direct consequence of this is that for smooth affine surfaces over algebraically closed fields, stably trivial spin torsors of rank $5$ are already trivial because the hypotheses imply $\widetilde{\op{CH}}^2(X)\cong\op{CH}^2(X)$. 

\begin{example}
\label{ex:so5}
We can first consider examples related to quadratic information in the Chow--Witt group. 
Consider the quadric $\op{Q}_4\cong\op{S}^2\wedge\mathbb{G}_{\op{m}}^{\wedge 2}$. We have
\[
\op{H}^2_{\op{Nis}}(\op{Q}_4,\mathbf{K}^{\op{MW}}_2)\cong \op{GW}(k).
\]
The projection map $\mathbf{K}^{\op{MW}}_2\to\mathbf{K}^{\op{M}}_2$ induces the dimension function $\op{GW}(k)\to\mathbb{Z}$ and any element in the kernel will give rise to a stably trivial spin torsor of rank $5$ over $\op{Q}_4$. This provides an algebraic realization and generalization of the $\op{SO}(3,2)$-bundles over $\op{S}^2$ coming from $\pi_2{\op{B}}\op{SO}(3,2)$ which are killed by stabilization to $\op{SO}(3,3)$. 

Consider the quadric $\op{Q}_5\cong\op{S}^2\wedge\mathbb{G}_{\op{m}}^{\wedge 3}$. We have
\[
\op{H}^2_{\op{Nis}}(\op{Q}_5,\mathbf{K}^{\op{MW}}_2)\cong \op{W}(k).
\]
The projection map $\mathbf{K}^{\op{MW}}_2\to\mathbf{K}^{\op{M}}_2$ induces the zero map, in particular any element of $\op{W}(k)$ gives rise to a stably trivial spin torsor of rank $5$ over $\op{Q}_5$. This provides an algebraic realization and generalization of the $\op{SO}(5)$-bundles over $\op{S}^5$ coming from $\pi_5{\op{B}}\op{SO}(5)$. 
\end{example}

\begin{example}
\label{ex:mkp2}
Let $k$ be an algebraically closed field of characteristic $\neq 2$. We can consider the $4$-dimensional smooth affine $k$-scheme $X$ with stably free rank $2$ vector bundle constructed by Mohan Kumar \cite{mohan:kumar}. 
By  \cite[Section 5]{mk}, the rank $2$ vector bundle is detected by a non-trivial class in $\ker\left(\widetilde{\op{CH}}^2(X)\to\op{CH}^2(X)\right)$. By Corollary~\ref{cor:stabso5}, this non-trivial class will correspond to a non-trivial stably trivial spin torsor of rank $5$ over $X$. It can be constructed by taking the stably free rank $2$ module, viewed as a (stably non-trivial) symplectic line bundle, add a trivial symplectic line and view the resulting $\op{Sp}_4$-torsor as spin lift of a quadratic form of rank $5$ via the sporadic isomorphism.
\end{example}

\begin{example}
\label{ex:ksp3}
Interesting examples of stably trivial spin torsors realizing the degree $3$ invariant in $\op{H}^3_{\op{Nis}}(X,\mathbf{KSp}_3)$ can be found over higher-dimensional schemes (but still of $\mathbb{A}^1$-homotopical dimension $3$). For instance, over the base field $k$, we have 
\[
\op{H}^3_{\op{Nis}}(\op{Q}_6,\mathbf{KSp}_3)\cong \op{Ch}^3(\op{Q}_6)\cong\op{H}^3(\op{Q}_6;\mathbf{K}^{\op{M}}_3/2)\cong \mathbb{Z}/2\mathbb{Z}.
\]
If $k$ is algebraically closed, this corresponds to the classical statement  $\pi_6({\op{B}}\op{SO}(5))\cong\mathbb{Z}/2\mathbb{Z}$. Note also that $\op{H}^1_{\op{Nis}}(\op{Q}_6,\mathbf{K}^{\op{MW}}_2)=0$, so these examples are not in the image of the looped Postnikov invariant map $\Omega k_3$.
\end{example}

\subsection{Stably trivial torsors of rank $3$ and $4$}

We begin by identifying the lifting classes of $\op{Spin}(3)$-torsors, via the sporadic isomorphism $\op{Spin}(3)\cong\op{SL}_2$. 

\begin{proposition}
\label{prop:classifyso3}
Let $k$ be an infinite perfect field of characteristic unequal to $2$, let $X=\op{Spec} A$ be a smooth affine scheme over $k$ of dimension $\leq 3$. We have the following statements:
\begin{enumerate}
\item There is an isomorphism $\bm{\pi}_2^{\mathbb{A}^1}{\op{B}}\op{Spin}(3)\cong\mathbf{K}^{\op{MW}}_2$.
Consequently, the second lifting class for a  torsor is the Euler class of the corresponding oriented rank $2$ vector bundle in  $\op{H}^2_{\op{Nis}}(X,\mathbf{K}^{\op{MW}}_2)\cong\widetilde{\op{CH}}^2(X)$. 
\item There are short exact sequences of strictly $\mathbb{A}^1$-invariant sheaves
\[
0\to \mathbf{T}_4'\to\bm{\pi}^{\mathbb{A}^1}_3{\op{B}}\op{Spin}(3)\to \mathbf{KSp}_3\to 0, \textrm{ and}
\]
\[
0\to\mathbf{D}_5\to\mathbf{T}_4'\to\mathbf{S}_4'\to 0
\]
where $\mathbf{D}_5$ is a quotient of $\mathbf{I}^5$ and the canonical morphism $\mathbf{K}^{\op{M}}_4/12\to\mathbf{S}_4'$ becomes an isomorphism after $3$-fold contraction. 

There is an exact sequence 
\[
\op{H}^1_{\op{Nis}}(X,\mathbf{K}^{\op{MW}}_2)\xrightarrow{\Omega k_3} \op{H}^3_{\op{Nis}}(X,\bm{\pi}^{\mathbb{A}^1}_3{\op{B}}\op{Spin}(3))\to [X,{\op{B}}\op{Spin}(3)]\xrightarrow{\op{e}}\widetilde{\op{CH}}^2(X)\to 0.
\]
The third lifting class in $\op{H}^3_{\op{Nis}}(X,\bm{\pi}^{\mathbb{A}^1}_3{\op{B}}\op{Spin}(3))$ decomposes into contributions from a class in $\mathbf{D}_5$-cohomology, cf. \cite{AsokFasel}, a motivic cohomology class in $\op{H}^3_{\op{Nis}}(X,\mathbf{K}^{\op{M}}_4/12)$, and  a mod 2 class in 
\[
\op{H}^3_{\op{Nis}}(X,\mathbf{KSp}_3)\cong \op{coker}\left(\op{Ch}^2(X)\xrightarrow{\op{Sq}^2}\op{Ch}^3(X)\right).
\]
\end{enumerate}
\end{proposition}

\begin{proof}
The results follow from the identification of $\op{Spin}(3)$ with $\op{SL}_2$ in Proposition~\ref{prop:so3} together with computations of $\mathbb{A}^1$-homotopy groups. Point (1) follows from \cite[Theorem 5.39]{MField}, the description in (2) is obtained in \cite[Theorem 3.3, Lemma 7.2]{AsokFasel}. The description of $\op{H}^3_{\op{Nis}}(X;\mathbf{KSp}_3)$ in terms of Steenrod operations is established in \cite[Proposition 4.16]{AsokFasel}. 
\end{proof}

\begin{remark}
Using Proposition~\ref{prop:so3}, the examples of quadratic forms corresponding to the torsors above can be constructed fairly explicitly. If $X=\op{Spec}A$ is a smooth affine scheme and $\mathcal{P}$ is an oriented projective module of rank 2 over $A$, then we can consider its bundle of orientation-preserving automorphisms which is the principal $\op{SL}_2$-bundle over $X$ such that the associated vector bundle for the standard representation is the original module $\mathcal{P}$. If we take the associated vector bundle for the $\op{SL}_2$-representation given by conjugation on trace $0$ matrices in $\op{Mat}_{2\times 2}(k)$, we get the required quadratic form of rank $3$. This sets up a bijection betweeen oriented projective modules of rank $2$ and rationally trivial quadratic forms of rank $3$. Note that the projective modules of rank $2$ can all be obtained by means of the Hartshorne--Serre construction from codimension 2 local complete intersections in $X$.
\end{remark}

\begin{corollary}
\label{cor:spin3}
Let $k$ be an algebraically closed field of characteristic unequal to $2$, and let $X$ be a smooth affine scheme over $k$ of dimension $\leq 3$. Then there is a bijection
\[
\op{H}^1_{\op{Nis}}(X,\op{Spin}(3))\cong \widetilde{\op{CH}}^2(X).
\]
\end{corollary}

\begin{proof}
Over an algebraically closed field, the top Chow group of a smooth affine scheme is uniquely divisible. In particular,  $\op{Ch}^3(X):=\op{CH}^3(X)/2$ is trivial. This implies that the contribution from $\mathbf{KSp}_3$-cohomology vanishes. The unique divisibility of the multiplicative group of an algebraically closed field implies that the group  $\op{H}^3_{\op{Nis}}(X,\mathbf{S}_4')\cong \op{H}^3_{\op{Nis}}(X,\mathbf{K}^{\op{M}}_4/12)$ is also trivial, cf.  \cite[Proposition 5.4]{AsokFasel}. Finally, the restriction of the sheaf $\mathbf{I}^5$ to a smooth affine scheme of dimension $\leq 4$ over an algebraically closed field is trivial. This implies that the contribution from $\mathbf{D}_5$-cohomology vanishes. The third lifting set $\op{H}^3_{\op{Nis}}(X,\bm{\pi}^{\mathbb{A}^1}_3{\op{B}}\op{Spin}(3))$ is therefore trivial. Since all the higher obstructions vanish because they live above the Nisnevich cohomological dimension of $X$, any lifting class in $\op{H}^2_{\op{Nis}}(X,\bm{\pi}^{\mathbb{A}^1}_2{\op{B}}\op{Spin}(3))$ can be uniquely extended to a map $X\to{\op{B}}\op{Spin}(3)$. The representability theorem~\ref{thm:representability} provides the required bijection between the second lifting set and the isomorphism classes of rank $3$ spin bundles. 
\end{proof}

\begin{example}
\label{ex:so3:1}
Any class in $\widetilde{\op{CH}}^2(X)$ yields a non-trivial spin torsor of rank $3$ over $X$. Particularly interesting in this situation are those in the kernel of the projection $\widetilde{\op{CH}}^2(X)\to\op{CH}^2(X)$. There are examples of such classes over $3$-dimensional smooth affine schemes over fields of the form $\overline{k}(T)$ as well as examples over $4$-dimensional smooth affine schemes over algebraically closed fields as discussed in Example~\ref{ex:mkp2}. 
\end{example}

\begin{example}
\label{ex:so3:2}
Interesting examples of stably trivial torsors realizing the degree $3$ invariants can be found over higher-dimensional schemes (but still of $\mathbb{A}^1$-homotopical resp. Nisnevich cohomological dimension $3$). 

For instance, over the base field $k$, we have
\begin{itemize}
\item $\op{H}^3_{\op{Nis}}(\op{Q}_6,\mathbf{D}_5)$ is a quotient of $\op{H}^3_{\op{Nis}}(\op{Q}_6,\mathbf{I}^5)\cong \op{I}^2(k)$,
\item $\op{H}^3_{\op{Nis}}(\op{Q}_6,\mathbf{K}^{\op{M}}_4/12)\cong k^\times/12$, and 
\item $\op{H}^3_{\op{Nis}}(\op{Q}_6,\mathbf{KSp}_3)\cong \op{Ch}^3(\op{Q}_6)\cong \mathbb{Z}/2\mathbb{Z}$.
\end{itemize}
If $k$ is algebraically closed, the first two of these vanish and the last one corresponds to the classical statement that $\pi_6({\op{B}}\op{SO}(3))\cong\mathbb{Z}/2\mathbb{Z}$. Note that the Nisnevich cohomology long exact sequences associated to the short exact sequences of strictly $\mathbb{A}^1$-invariant sheaves from Proposition~\ref{prop:classifyso3} reduce to short exact sequences because $\op{Q}_6$ has only non-trivial Nisnevich cohomology in degrees $0$ and $3$. So any non-trivial class in the above sets actually gives a non-trivial lifting class in $\op{H}^3_{\op{Nis}}(\op{Q}_6,\bm{\pi}_3^{\mathbb{A}^1}{\op{B}}\op{Spin}(3))$. 

Similarly, for $\op{Q}_7$, we have 
\begin{itemize}
\item $\op{H}^3_{\op{Nis}}(\op{Q}_7,\mathbf{D}_5)$ is a quotient of $\op{H}^3_{\op{Nis}}(\op{Q}_7,\mathbf{I}^5)\cong \op{I}(k)$,
\item $\op{H}^3_{\op{Nis}}(\op{Q}_7,\mathbf{K}^{\op{M}}_4/12)\cong \mathbb{Z}/12$, and 
\item $\op{H}^3_{\op{Nis}}(\op{Q}_7,\mathbf{KSp}_3)\cong 0$ as in the proof of \cite[Lemma 7.3]{AsokFasel}.
\end{itemize}
The second item on the list corresponds to the classical fact that $\pi_7{\op{B}}\op{SO}(3)\cong\mathbb{Z}/12\mathbb{Z}$. 
\end{example}

\begin{example}
\label{ex:so3:3}
One more type of interesting examples related to $\mathbf{K}^{\op{M}}_4/12$ should be mentioned. If $k$ is an algebraically closed field, the construction of Mohan Kumar produces a $4$-dimensional smooth affine scheme $X$ over $k(T)$ which has a non-trivial class in $\op{H}^3_{\op{Nis}}(X,\mathbf{K}^{\op{M}}_4/12)$, detected on $\op{CH}^4(X)/3$. Mohan Kumar produced a stably free module of rank $3$ from this. In \cite{mk}, a variation of Mohan Kumar's construction was shown to produce a stably free module of rank $2$ over $X$ (which stabilizes to Mohan Kumar's example). In our context, this rank $2$ stably free module produces a non-trivial spin torsor of rank $3$ over $X$ (after base change to a finite purely inseparable extension of $k(T)$). Clearing denominators, we find that there are examples of quadratic forms of rank $3$ over $5$-dimensional smooth affine schemes over algebraically closed fields detected in $\op{H}^3_{\op{Nis}}(X,\mathbf{K}^{\op{M}}_4/12)$. 
\end{example}

Now that we have discussed classification and examples of torsors of rank $3$, we turn to the rank $4$ case. 

\begin{proposition}
\label{prop:classifyso4}
Let $k$ be an infinite perfect field of characteristic unequal to $2$, let $X=\op{Spec} A$ be a smooth affine scheme over $k$. There is a bijection
\[
\op{H}^1_{\op{Nis}}(X;\op{Spin}(4))\cong \op{H}^1_{\op{Nis}}(X;\op{Spin}(3))\times \op{H}^1_{\op{Nis}}(X;\op{Spin}(3)).
\]

On the level of lifting classes, the stabilization morphism $\op{Spin}(3)\to\op{Spin}(4)$ is given by the diagonal embedding. In particular, every spin torsor of rank $3$ which becomes trivial as a spin torsor of rank $4$ is already trivial. 
\end{proposition}

\begin{proof}
This follows directly from Propositions~\ref{prop:so4} and \ref{prop:stab34}. 
\end{proof}

\begin{remark}
\label{rem:test}
It is now straightforward to get examples of rank $4$ torsors by taking pairs of the previously discussed examples \ref{ex:so3:1}, \ref{ex:so3:2} and \ref{ex:so3:3} of rank $3$ torsors. 
\end{remark}

Now we want to discuss the stabilization to $\op{Spin}(5)$. 

\begin{proposition}
\label{prop:stab45pi}
\begin{enumerate}
\item
The map 
\[
\bm{\pi}^{\mathbb{A}^1}_n{\op{B}}\op{SL}_2\times\bm{\pi}^{\mathbb{A}^1}_n {\op{B}}\op{SL}_2\xrightarrow{\cong} \bm{\pi}^{\mathbb{A}^1}_n{\op{B}}\op{Spin}(4)\to \bm{\pi}^{\mathbb{A}^1}_n{\op{B}}\op{Spin}(5)\xrightarrow{\cong} 
\bm{\pi}^{\mathbb{A}^1}_n{\op{B}}\op{Sp}_4
\]
induced by the stabilization morphism $\op{Spin}(4)\to\op{Spin}(5)$  and the sporadic isomorphisms is given by the sum of the stabilization morphisms $\bm{\pi}^{\mathbb{A}^1}_n{\op{B}}\op{SL}_2\to \bm{\pi}^{\mathbb{A}^1}_n{\op{B}}\op{Sp}_4$ in the symplectic series. 
\item
The map 
\[
\bm{\pi}^{\mathbb{A}^1}_n{\op{B}}\op{SL}_2\xrightarrow{\cong} \bm{\pi}^{\mathbb{A}^1}_n{\op{B}}\op{Spin}(3)\to \bm{\pi}^{\mathbb{A}^1}_n{\op{B}}\op{Spin}(5)\xrightarrow{\cong} 
\bm{\pi}^{\mathbb{A}^1}_n{\op{B}}\op{Sp}_4
\]
induced by the stabilization morphism $\op{Spin}(3)\to\op{Spin}(5)$ is twice the morphism induced from stabilization $\op{SL}_2\to\op{Sp}_4$. 
\end{enumerate}
\end{proposition}

\begin{proof}
For (1), it is clear that the morphism is the sum of the restrictions to the individual factors. By Proposition~\ref{prop:stab45}, the morphism $\op{SL}_2\times\op{SL}_2\to\op{Sp}_4$ is given by the long-root embedding. In particular, both morphisms are stabilization morphisms; one by adding an identity matrix in the lower right corner, one by adding an identity matrix in the upper left. The first one is the usual stabilization, and it remains to show that the other one is homotopic to the first one. Put differently, we want to show that the two embeddings $\op{SL}_2\to\op{Sp}_4$ via the two choices of long-root embeddings are $\mathbb{A}^1$-homotopic. But one of the long-root embedding is converted into the other by an appropriate conjugation with an element of the Weyl group. The standard representatives of elements of the Weyl group have explicit elementary factorizations (by definition), which provides the required chain of naive $\mathbb{A}^1$-homotopies connecting the two long-root embeddings. This shows (1). 

Statement (2) follows from (1) together with the assertion of Proposition~\ref{prop:stab34} that stabilization of spin groups corresponds to the composition $\op{SL}_2\xrightarrow{\Delta}\op{SL}_2\times\op{SL}_2\to\op{Sp}_4$. The first map induces the diagonal on $\mathbb{A}^1$-homotopy sheaves and the second takes the sum by (1). 
\end{proof}

We recall the effect of the symplectic stabilization $\op{SL}_2\to\op{Sp}_4$ on $\mathbb{A}^1$-homotopy sheaves. 

\begin{proposition}
\label{prop:stabsl2sp4}
\begin{enumerate}
\item The morphism 
\[
\mathbf{K}^{\op{MW}}_2\cong\bm{\pi}^{\mathbb{A}^1}_2{\op{B}}\op{SL}_2\to \bm{\pi}^{\mathbb{A}^1}_2{\op{B}}\op{Sp}_4 \cong\mathbf{K}^{\op{MW}}_2
\]
induced from symplectic stabilization $\op{SL}_2\to\op{Sp}_4$ is the identity, when we identify $\mathbf{K}^{\op{MW}}_2$ of fields with second group cohomology of the discrete groups. 
\item The morphism 
\[
\bm{\pi}^{\mathbb{A}^1}_3{\op{B}}\op{SL}_2\to \bm{\pi}^{\mathbb{A}^1}_3{\op{B}}\op{Sp}_4 \cong\mathbf{KSp}_3
\]
induced from symplectic stabilization $\op{SL}_2\to\op{Sp}_4$ is the natural projection in the exact sequence in point (2) of Proposition~\ref{prop:classifyso3}. 
\end{enumerate}
\end{proposition}

\begin{proof}
The first one follows from the symplectic stabilization results in \cite{torsors}, the second one follows from the computations in \cite{AsokFasel}. 
\end{proof}

\begin{proposition}
\label{prop:stab345}
\begin{enumerate}
\item The morphism 
\[
\mathbf{K}^{\op{MW}}_2\times\mathbf{K}^{\op{MW}}_2\cong\bm{\pi}^{\mathbb{A}^1}_2{\op{B}}\op{Spin}(4)\to \bm{\pi}^{\mathbb{A}^1}_2{\op{B}}\op{Spin}(5) \cong\mathbf{K}^{\op{MW}}_2
\]
induced from orthogonal stabilization $\op{Spin}(4)\to\op{Spin}(5)$ is the sum of the identities on the two factors, when we identify  $\mathbf{K}^{\op{MW}}_2$ of fields with second group cohomology of the discrete groups. 
With this identification,
\[
\mathbf{K}^{\op{MW}}_2\cong\bm{\pi}^{\mathbb{A}^1}_2{\op{B}}\op{Spin}(3)\to \bm{\pi}^{\mathbb{A}^1}_2{\op{B}}\op{Spin}(5) \cong\mathbf{K}^{\op{MW}}_2
\]
induced from orthogonal stabilization $\op{Spin}(3)\to\op{Spin}(5)$ is multiplication by $2$. 
\item The morphism 
\[
\bm{\pi}^{\mathbb{A}^1}_3{\op{B}}\op{SL}_2\times \bm{\pi}^{\mathbb{A}^1}_3{\op{B}}\op{SL}_2\cong \bm{\pi}^{\mathbb{A}^1}_3{\op{B}}\op{Spin}(4)\to \bm{\pi}^{\mathbb{A}^1}_3{\op{B}}\op{Spin}(5) \cong\mathbf{KSp}_3
\]
induced from orthogonal stabilization $\op{Spin}(3)\to\op{Spin}(5)$ is the sum of the natural projection of Proposition~\ref{prop:classifyso3} on each of the two factors. 
Similarly, 
\[
\bm{\pi}^{\mathbb{A}^1}_3{\op{B}}\op{SL}_2\cong\bm{\pi}^{\mathbb{A}^1}_3{\op{B}}\op{Spin}(3)\to \bm{\pi}^{\mathbb{A}^1}_3{\op{B}}\op{Spin}(5) \cong\mathbf{KSp}_3
\]
induced from orthogonal stabilization $\op{Spin}(3)\to\op{Spin}(5)$ is twice  the natural projection of Proposition~\ref{prop:classifyso3} on each of the two factors. 
\end{enumerate}
\end{proposition}

\begin{proof}
This is a combination of Propositions~\ref{prop:stab45pi} and \ref{prop:stabsl2sp4}. 
\end{proof}

\begin{remark}
The above statements can, via complex realization, be compared with the classical statements on the stabilization of the homotopy of the (compact) special orthogonal groups, cf. \cite{steenrod}. Classically, we have the following diagram
\[
\xymatrix{
\pi_3(\op{SO}(3))\ar[r] \ar[d]_\cong & \pi_3(\op{SO}(4)) \ar[r] \ar[d]^\cong & \pi_3(\op{SO}(5))\ar[d]^\cong \\
\mathbb{Z}\ar[r]_\Delta & \mathbb{Z}\oplus\mathbb{Z}\ar[r]_f & \mathbb{Z}.
}
\]
where the map $f$ is given by $(1,1)\mapsto 2$, $(1,0)\mapsto 1$. The first generator is the one given by the image of the generator from $\op{SO}(3)$ (i.e. it is realized by the conjugation action of the unit quaternions on themselves), the second generator is given by left multiplication of unit quaternions on all quaternions. 

The above computations reproduce exactly this picture. Actually, the development of the sporadic isomorphisms in Section~\ref{sec:sporadic} can be used to reprove the classical statements in a different manner with significantly less homotopical arguments. Over $\mathbb{C}$, the sequence
\[
[\op{Q}_4,{\op{B}}\op{SO}(3)]_{\mathbb{A}^1}\to [\op{Q}_4,{\op{B}}\op{SO}(4)]_{\mathbb{A}^1}\to [\op{Q}_4,{\op{B}}\op{SO}(5)]_{\mathbb{A}^1}
\]
reproduces exactly the classical sequence above by noting that $\op{H}^2_{\op{Nis}}(\op{Q}_4,\mathbf{K}^{\op{MW}}_2)\cong \op{GW}(k)$. The classical description of the generators $\op{Q}_3\to\op{SL}_2$ and $\op{Q}_3\to\op{SL}_2\times\op{SL}_2$ also follows from the statements in Section~\ref{sec:sporadic}.
\end{remark}

We now have all the homotopical information to discuss stabilization of torsors of rank $3$ and $4$ and provide some examples. 

\begin{proposition}
Let $k$ be a field of characteristic $\neq 2$ and let $X$ be a smooth affine scheme of dimension $\leq 3$ over $k$. 
\begin{enumerate}
\item A spin torsor of rank $3$ over $X$ is stably trivial if and only if its lifting class in $\widetilde{\op{CH}}^2(X)$ has $2$-torsion image in $\op{CH}^2(X)$. 
\item A spin torsor of rank $4$ over $X$ classified by $(\gamma,\delta)\in\widetilde{\op{CH}}^2(X)\times \widetilde{\op{CH}}^2(X)$ is stably trivial if and only if the class $\gamma+\delta$ has trivial image in $\op{CH}^2(X)$. 
\end{enumerate}
\end{proposition}

\begin{proof}
Follows directly from Corollaries~\ref{cor:spin3} and \ref{cor:stabso5} as well as Propositions~\ref{prop:classifyso4} and \ref{prop:stab345}. 
\end{proof}

We therefore get the following examples of stably trivial spin torsors of ranks $3$ and $4$ related to the lifting class in the second Chow--Witt group.

\begin{example}
Let $k$ be a field and let $X$ be a scheme of dimension $\leq 3$ with a non-trivial class $\alpha\in\widetilde{\op{CH}}^2(X)$. Then $(\alpha,-\alpha)\in\op{H}^2_{\op{Nis}}(X,\bm{\pi}^{\mathbb{A}^1}_2{\op{B}}\op{Spin}(4))$ gives a non-trivial stably trivial torsor of rank $4$ over $X$. 

More complicated examples of stably trivial torsors of rank $3$ or $4$ arising from the kernel of $\widetilde{\op{CH}}^2(X)\to\op{CH}^2(X)$ can be manufactured as in Example~\ref{ex:mkp2}. 

Finally, we can get examples related to the prime $2$. Over $\mathbb{R}$, the complement $X$ of the conic $U^2+V^2+W^2=0$ in $\mathbb{P}^2$ is a smooth affine scheme with $\op{CH}^2(X)\cong\mathbb{Z}/2\mathbb{Z}$. We can lift the class of the $k$-rational point along $\widetilde{\op{CH}}^2(X)\twoheadrightarrow\op{CH}^2(X)$ and consider the torsor of rank $3$ associated to this element, viewed as lifting class in $\op{H}^2_{\op{Nis}}(X,\bm{\pi}^{\mathbb{A}^1}_2{\op{B}}\op{Spin}(3))$. The resulting bundle will be non-trivial but stably trivial because its image in the lifting set $\op{H}^2_{\op{Nis}}(X,\bm{\pi}^{\mathbb{A}^1}_2{\op{B}}\op{Spin}(6))\cong\op{CH}^2(X)$ will be twice the generator, by the stabilization results above. 
\end{example}

\begin{example}
Any combination of the examples of quadratic forms related to degree $3$ invariants, cf. Examples~\ref{ex:so3:2} and \ref{ex:so3:3}, will result in a stably trivial torsor because the degree $3$ lifting classes are not stably visible. However, torsors of rank $3$ over smooth affine schemes $X$ of homotopical dimension $3$ and with trivial characteristic class in $\widetilde{\op{CH}}^2(X)$ will already become hyperbolic by adding a single hyperbolic plane. Any lifting class not related to $\mathbf{KSp}_3$-cohomology will be invisible anyway by Proposition~\ref{prop:stab345}. On the other hand, the $\mathbf{KSp}_3$-class after stabilization by a hyperbolic plane will be twice the projection of the class of the associated projective rank $2$ module to $\op{H}^3_{\op{Nis}}(X,\mathbf{KSp}_3)$. But the latter is $2$-torsion.
\end{example}

\section{Spin torsors vs quadratic forms}
\label{sec:quad}

Now that we have studied in detail the classification of stably trivial spin torsors, we investigate the relation between spin torsors and quadratic forms. The main point is that the classification of rationally hyperbolic quadratic forms can be obtained from the classification of spin torsors by dividing out the action of the fundamental group $\bm{\pi}_1^{\mathbb{A}^1}{\op{B}}_{\op{Nis}}\op{SO}(n)$. In particular, most of the stably trivial spin torsors discussed in Section~\ref{sec:dim3} provide examples of non-trivial stably trivial quadratic forms.

\subsection{Quadratic forms vs spin torsors}
At this point, we aim to say something about the group $\op{H}^1_{\op{Nis}}(X,\op{SO}(n))$ of rationally trivial $\op{SO}(n)$-torsors over smooth affine schemes $X$. To this end, we want to discuss the relation between the two classification problems induced by the natural map 
\[
[X,{\op{B}}_{\op{Nis}}\op{Spin}(n)]_{\mathbb{A}^1}\to [X,{\op{B}}_{\op{Nis}}\op{SO}(n)]_{\mathbb{A}^1}.
\]
The spin cover $\op{Spin}(n)\to\op{SO}(n)$ is a finite \'etale (even Galois) map of degree $2$ (over a field of characteristic $\neq 2$) and therefore, by \cite[Lemma 6.5]{MField}, it is an $\mathbb{A}^1$-covering space. On the level of homotopy groups, this implies the following statements: first, there are induced isomorphisms
\[
\bm{\pi}^{\mathbb{A}^1}_i\op{Spin}(n)\cong \bm{\pi}^{\mathbb{A}^1}_i\op{SO}(n)
\]
for $i\geq 2$. Furthermore, since the group $\op{Spin}(n)$ is generated by unipotent matrices, it is $\mathbb{A}^1$-connected and there is an extension
\[
1\to \bm{\pi}^{\mathbb{A}^1}_1\op{Spin}(n)\to \bm{\pi}^{\mathbb{A}^1}_1\op{SO}(n)\to \mu_2\to 0.
\]
Finally, the spin covering induces a bijection of pointed sets
\[
\bm{\pi}^{\mathbb{A}^1}_0\op{SO}(n)\cong \mathscr{H}^1_{\et}(\mu_2),
\]
where the target denotes the Nisnevich sheafification associated to the presheaf $X\mapsto \op{H}^1_{\et}(X,\mu_2)$. Similar statements for the groups $\op{PGL}_n$ (in particular for $\op{SO}(3)$) have been discussed in \cite[Section 3]{AKW15}. All in all, we have an $\mathbb{A}^1$-fiber sequence
\[
{\op{B}}_{\et}\mu_2\to{\op{B}}_{\op{Nis}}\op{Spin}(n)\to {\op{B}}_{\op{Nis}}\op{SO}(n),
\]
which is the restriction of a similar fiber sequence for \'etale classifying spaces along the natural map ${\op{B}}_{\op{Nis}}\op{SO}(n)\to {\op{B}}_{\et}\op{SO}(n)$. 

As in \cite[Section 3]{AKW15}, if $X$ is a smooth affine scheme, mapping $X$ to the above $\mathbb{A}^1$-fiber sequence yields an exact sequence of groups and pointed sets
\[
[X,\op{SO}(n)]_{\mathbb{A}^1}\to \op{H}^1_{\et}(X,\mu_2)\to [X,{\op{B}}_{\op{Nis}}\op{Spin}(n)]_{\mathbb{A}^1}\to [X,{\op{B}}_{\op{Nis}}\op{SO}(n)]_{\mathbb{A}^1}.
\]
We first want know when the last map is a surjection. Using the relative obstruction theory, a map $X=\op{Spec}A \to{\op{B}}_{\op{Nis}}\op{SO}(n)$ classifying a rationally hyperbolic quadratic form on $A$ has a spin lift if the sequence of obstruction classes in $\op{H}^{i+1}_{\op{Nis}}(X,\bm{\pi}^{\mathbb{A}^1}_i{\op{B}}_{\et}\mu_2)$ vanishes. At each stage, there is a choice of lifts  given by $\op{H}^{i}_{\op{Nis}}(X,\bm{\pi}^{\mathbb{A}^1}_i{\op{B}}_{\et}\mu_2)$. Note that the only non-trivial homotopy groups of ${\op{B}}_{\et}\mu_2$ are 
\[
\bm{\pi}^{\mathbb{A}^1}_1{\op{B}}_{\et}\mu_2\cong\mu_2, \textrm{ and } \bm{\pi}^{\mathbb{A}^1}_0{\op{B}}_{\et}\cong\mathscr{H}^1_{\et}(\mu_2).
\]
Since $\op{H}^i_{\op{Nis}}(X,\mu_2)$ is trivial for any smooth affine scheme $X$ and $i\geq 1$, the only relevant obstruction and lifting groups are those associated to $\mathscr{H}^1_{\et}(\mu_2)$. In particular, a spin lift exists if the obstruction in $\op{H}^1_{\op{Nis}}(X,\mathscr{H}^1_{\et}(\mu_2))$ vanishes. One could consider this obstruction class as an algebraic version of the second Stiefel--Whitney class $\op{w}_2$ of the quadratic form. In general, there is no reason for the group $\op{H}^1_{\op{Nis}}(X,\mathscr{H}^1_{\et}(\mu_2))$ to vanish, for instance we have $\op{H}^1_{\op{Nis}}(\op{Q}_2,\mathscr{H}^1_{\et}(\mu_2))\cong\mu_2$. However, the stabilization morphisms $\op{SO}(n)\to\op{SO}(n+1)$ induce isomorphisms on connected components for all $n\geq 3$ and therefore the obstruction class in $\op{H}^1_{\op{Nis}}(X,\mathscr{H}^1_{\et}(\mu_2))$ is a stable invariant. In particular, it vanishes for stably trivial quadratic forms, meaning that stably trivial quadratic forms will always have spin lifts. In the remainder of the section, we will concentrate on rationally hyperbolic quadratic forms admitting a spin lift. 

At this point, the exactness of the sequence of groups and pointed sets at the point $[X,{\op{B}}_{\op{Nis}}\op{Spin}(n)]_{\mathbb{A}^1}$ means that the image of the map $[X,{\op{B}}_{\op{Nis}}\op{Spin}(n)]_{\mathbb{A}^1}\to [X,{\op{B}}_{\op{Nis}}\op{SO}(n)]_{\mathbb{A}^1}$ of pointed sets is given by the orbit set
\[
[X,{\op{B}}_{\op{Nis}}\op{Spin}(n)]/\op{H}^1_{\et}(X,\mu_2).
\]
Reformulated, the isomorphism classes of rationally hyperbolic quadratic forms of rank $n$ over $X$ which admit a spin lift are given as equivalence classes of $\op{Spin}(n)$-torsors by the action of the degree 2 covers of $X$. 

\subsection{Action of line bundles}
\label{sec:action}
We now want to make the action of $\op{H}^1_{\et}(X,\mu_2)$  more precise to be able to compute its orbits on the isomorphism classes of spin torsors.

First we consider the cases of rank $3$ and $4$ and describe the action of degree 2 covers based on the sporadic isomorphisms. Using the sporadic isomorphism $\op{Spin}(3)\cong\op{SL}_2$ from Proposition~\ref{prop:so3}, we find that the spin covering $\op{Spin}(3)\to\op{SO}(3)$ can be described as the degree 2 covering $\op{SL}_2\to\op{PGL}_2$. In \cite[Section 3]{AKW15}, the identification 
\[
[X,{\op{B}}\op{GL}_2]/\op{H}^1_{\op{Nis}}(X,\mathbb{G}_{\op{m}})\cong [X,{\op{B}}_{\op{Nis}}\op{PGL}_2]
\]
was described explicitly: the action of $\op{H}^1_{\op{Nis}}(X,\mathbb{G}_{\op{m}})\cong\op{Pic}(X)$ on $[X,{\op{B}}\op{GL}_2]$ is given by twisting the rank 2 vector bundles by line bundles, and the $\op{PGL}_2$-torsors are then orbits of rank 2 vector bundles by twists with line bundles. Restricting this to $\op{SL}_2$-torsors, we get a description of the action of $\op{H}^1_{\et}(X,\mu_2)$ on $\op{Spin}(3)$-torsors; the cocartesian square
\[
\xymatrix{
\mu_2\ar[r] \ar[d] & \mathbb{G}_{\op{m}}\ar[d] \\
\op{SL}_2\ar[r] & \op{GL}_2
}
\]
implies that the action of $\op{H}^1_{\et}(X,\mu_2)$ on $[X,{\op{B}}\op{SL}_2]$ factors through the quotient 
\[
\op{H}^1_{\et}(X,\mu_2)\twoheadrightarrow \ker\left(2:\op{Pic}(X)\to\op{Pic}(X)\right)
\] 
associated to the sequence $1\to\mu_2\to\mathbb{G}_{\op{m}}\xrightarrow{2}\mathbb{G}_{\op{m}}\to 1$. Explicitly, an element of $\op{H}^1_{\et}(X,\mu_2)$ acts on $[X,{\op{B}}\op{SL}_2]$ by twisting with the associated $2$-torsion line bundle. Similarly, the action of $\op{H}^1_{\et}(X,\mu_2)$ on $\op{Spin}(4)$-torsors is given by using the sporadic isomorphism $\op{Spin}(4)\cong\op{SL}_2\times\op{SL}_2$: an element of $\op{H}^1_{\et}(X,\mu_2)$ acts on $[X,{\op{B}}\op{SL}_2\times{\op{B}}\op{SL}_2]$ by twisting both rank 2 bundles simultaneously with the associated $2$-torsion line bundle. 

A similar argument can be made for the other two sporadic cases. In rank $6$, we have a sequence of degree 2 coverings $\op{SL}_4\cong\op{Spin}(6)\to \op{SO}(6)\to \op{PGL}_4$. As in \cite[Section 3]{AKW15}, we can identify the action of line bundles on $[X,{\op{B}}\op{GL}_4]$ as the twisting. As done above in rank 3, this implies that the action of $\op{H}^1_{\et}(X,\mu_2)$ on $[X,{\op{B}}_{\op{Nis}}\op{Spin}(6)]$ is given by twists with $2$-torsion line bundles. Under the sporadic isomorphisms, the restriction of this action to $\op{Sp}_4\cong\op{Spin}(5)$ deals with the remaining case.

\subsection{Result and examples concerning quadratic forms}
After all the preparations, we are now ready to discuss the isomorphism classes of stably trivial quadratic forms, or more generally rationally hyperbolic quadratic forms which admit a spin lift. We know that the latter are given by the orbit set $[X,{\op{B}}_{\op{Nis}}\op{Spin}(n)]/\op{H}^1_{\et}(X,\mu_2)$, where the action is given by twists with $2$-torsion line bundles. It remains to revisit the general classification results and specific examples from Section~\ref{sec:dim3} to see what these results say about stably trivial quadratic forms. 

First, there is a generic remark. Having identified that the action of $\op{H}^1_{\et}(X,\mu_2)$ factors through an action of ${}_2\op{Pic}(X)$, there is a number of cases, relevant for our examples, in which the action will be trivial. Our main examples of rationally or stably trivial spin torsors in Section~\ref{sec:dim3} lived over smooth affine quadrics $\op{Q}_n$ or varieties constructed by Mohan Kumar in \cite{mohan:kumar}. For $n\geq 3$, we  have $\op{CH}^1(\op{Q}_n)/2=0$. The examples of Mohan Kumar for odd primes $p$ are open subvarieties of a hypersurface complement $\mathbb{P}^{p+1}\setminus Z$ where $Z$ has degree a power of $p$. In particular, these will also have  trivial $\op{CH}^1/2$. In these cases, the classification of rationally hyperbolic forms admitting a spin lift agrees with the classification of $\op{Spin}(n)$-torsors. 

\begin{example}
The above remark applies to Examples~\ref{ex:so5}, \ref{ex:ksp3}, \ref{ex:so3:2} and \ref{ex:so3:3}. In all these cases, we get examples of non-trivial stably trivial quadratic forms.
\end{example}

We consider the classification of rationally trivial quadratic forms of rank $6$. 
\begin{proposition}
\label{prop:nso6}
Let $k$ be an infinite perfect field of characteristic unequal to $2$, let $X=\op{Spec}A$ be a smooth  affine scheme over $k$ of dimension $\leq 3$. The action of $\op{H}^1_{\et}(X,\mu_2)$ on the lifting classes for $\op{Spin}(6)$-torsors is induced from the standard action of $\op{Pic}(X)$ on the Chern classes of vector bundles: 
\begin{eqnarray*}
\op{c}_2&\mapsto& \op{c}_2+6\ell^2\\
\op{c}_3&\mapsto& \op{c}_3+4\ell^3
\end{eqnarray*}
In particular, rationally hyperbolic quadratic forms of rank $6$ over $X$ are given by orbits of oriented rank 4 vector bundles modulo twists by $2$-torsion line bundles. Such a form is stably trivial if the image of its second Chern class in $\op{CH}^2(X)/{}_2\op{Pic}(X)$ equals the orbit of the trivial bundle. 
\end{proposition}

\begin{proof}
The statements about the action being given by twist with a line bundle follow from the previous discussion in Section~\ref{sec:action}. An oriented rank 4 vector bundle is a direct sum of an oriented rank 3 bundle and a trivial line, by Serre's splitting and the dimension assumption. The following formulas above reflect what happens to the Chern classes of such a bundle under twist with a line bundle of class $\ell$: 
\begin{eqnarray*}
\op{c}_2&\mapsto& \op{c}_2+3\op{c}_1\ell+6\ell^2\\
\op{c}_3&\mapsto& \op{c}_3+2\op{c}_2\ell+3\op{c}_1\ell^2+4\ell^3
\end{eqnarray*}
The term $2\op{c}_2\ell$ can be omitted because $\ell$ is $2$-torsion in our case. The contributions related to $\op{c}_1$ vanish because we have oriented bundles.
\end{proof}

\begin{example}
In the case of Example~\ref{ex:so6}, we have $\op{Pic}(X)=\op{Pic}(\overline{X})/[\overline{X}\setminus X]$. Now $\op{Pic}(\overline{X})$ is an extension of $\op{Pic}^0(\overline{X})$ and $\op{NS}(X)$. The latter is finitely generated and the former is an abelian variety. In particular, the $2$-torsion in $\op{Pic}(X)$ is finite. Consequently, Example~\ref{ex:so6} provides uncountably many isomorphism classes of stably trivial quadratic forms of rank $6$. 
\end{example}

\begin{remark}
\label{rem:so6hyp}
The composition $\op{SL}_3\to\op{SL}_4\to\op{SO}(6)$ of the stabilization and the sporadic isogeny is the hyperbolic morphism, cf. \cite[Propositions 2.3.1 and 2.3.4]{octonion}. By dimension reasons, for a scheme $X$ of dimension $\leq 3$, every quadratic bundle on $X$ is the hyperbolic bundle of some rank $3$ oriented vector bundle. In particular, the examples considered above are not only stably hyperbolic, they are all hyperbolic. 
\end{remark}

The technique above applies more generally. If the variety $X$ arises as hypersurface complement in $\mathbb{P}^n$, then the $2$-torsion in the Picard group will be finite. If we find ourselves in a situation where one of the lifting groups for stably trivial $\op{Spin}$-torsors happens to be infinite, there will automatically be infinitely many isomorphism classes of stably trivial quadratic forms. 

\begin{remark}
For the other cases, the action of ${}_2\op{Pic}(X)$ on the lifting groups is not so easy to identify. It seems likely that the action $\widetilde{\op{CH}}^2(X)$ can be described as follows: there is an exact sequence 
\[
\op{CH}^2(X)\to\widetilde{\op{CH}}^2(X)\to\op{H}^2_{\op{Nis}}(X,\mathbf{I}^2)\to 0.
\]
Now the action of an element of ${}_2\op{Pic}(X)$ on $\op{CH}^2(X)$ is given as in Proposition~\ref{prop:nso6}, and the action on $\op{H}^2_{\op{Nis}}(X,\mathbf{I}^2)$ should be given by addition with the image of the class in ${}_2\op{Pic}(X)$ under the boundary map ${}_2\op{Pic}(X)\subset \op{CH}^1(X)\to \op{H}^2_{\op{Nis}}(X,\mathbf{I}^2)$. At the moment, I cannot make this more precise, but probably the Hartshorne--Serre correspondence for rank $2$ vector bundles allows to identify exactly the action of line bundle twists on oriented vector bundles. 
Anyway, it is not clear at this point if the Example~\ref{ex:mkp2} is actually an example of a non-trivial quadratic form (or just an example of an interesting spin torsor) since the lifting classes detecting non-triviality of the spin torsor are $2$-torsion and the Picard group has a non-trivial $2$-torsion element. 

Similarly, the action of twisting by $2$-torsion line bundles on the cohomology groups $\op{H}^3_{\op{Nis}}(X,\bm{\pi}^{\mathbb{A}^1}_3({\op{B}}_{\op{Nis}}\op{Spin}(n)))$ needs to be made explicit to get more detailed results on the classification of rationally hyperbolic quadratic forms. 
\end{remark}

In any case, the combination of the results in Section~\ref{sec:dim3} and the computation in Proposition~\ref{prop:nso6} imply the following general result: 

\begin{proposition}
Let $k$ be an infinite perfect field of characteristic $\neq 2$, and let $X=\op{Spec}A$ be a smooth affine scheme over $k$ of dimension $\leq 3$. A rationally hyperbolic quadratic form over $A$ is stably trivial if and only if the image of its second Chern class in $\op{CH}^2(X)/{}_2\op{Pic}(X)$ is in the orbit of the hyperbolic form, where the action of line bundle class $\ell$ is given by $x\mapsto x+6\ell^2$. 
\end{proposition}

It seems that the action of ${}_2\op{Pic}(X)$ would always be an additive action by $2$-torsion elements on the Nisnevich cohomology groups. If true, this would generally imply that spin torsors corresponding to lifting classes which are not $2$-torsion always give non-trivial rationally hyperbolic quadratic forms. 

There is one more generic class of stably trivial quadratic forms in rank $4$. 

\begin{proposition}
\label{prop:genso4}
Let $k$ be an infinite perfect field of characteristic $\neq 2$, and let $X=\op{Spec}A$ be a smooth affine scheme over $k$ of dimension $\leq 3$. Let $\alpha\in\widetilde{\op{CH}}^2(X)$ be a class. Then the lifting class $(\alpha,-\alpha)\in \op{H}^2_{\op{Nis}}(X,\bm{\pi}^{\mathbb{A}^1}_2{\op{B}}\op{Spin}(4))$ induces a non-trivial stably trivial form if $\alpha$ is non-trivial in $\op{CH}^2(X)/{}_2\op{Pic}(X)$. 
\end{proposition}

The quadratic bundles in the proposition are stably hyperbolic, because every bundle of rank $6$ is hyperbolic as discussed in Remark~\ref{rem:so6hyp}. Using the identification of the hyperbolic morphism in rank $4$, cf. Proposition~\ref{prop:hypso4}, we can actually say something about when these bundles are hyperbolic. 

\begin{proposition}
Let $k$ be an infinite perfect field of characteristic $\neq 2$, and let $X=\op{Spec}A$ be a smooth affine scheme over $k$ of dimension $\leq 3$. Under the bijection of Proposition~\ref{prop:classifyso4} and the explicit identification of the hyperbolic morphism in Proposition~\ref{prop:hypso4}, a rank $4$ spin torsor in $\op{H}^1_{\op{Nis}}(X,\op{Spin}(4))\cong\op{H}^1_{\op{Nis}}(X,\op{SL}_2)\times\op{H}^1_{\op{Nis}}(X,\op{SL}_2)$ is hyperbolic if and only if the first component in the product decomposition is trivial. 
\end{proposition}

This implies the existence of many stably hyperbolic but non-hyperbolic quadratic forms of rank $4$ over schemes of dimension $\leq 3$. 

\begin{example}
\label{ex:nonhyp}
We can take one of the $\op{SL}_2$-torsors discussed in Examples~\ref{ex:so3:2} or \ref{ex:so3:3}. Let $\alpha\in\op{H}^1(X,\op{SL}_2)$ be one such torsor, and let $\beta\in\op{H}^1(X,\op{SL}_2)$ be any other torsor. Then the torsor corresponding to the element 
\[
(\alpha,\beta)\in\op{H}^1(X,\op{SL}_2)\times\op{H}^1(X,\op{SL}_2)\cong\op{H}^1(X,\op{Spin}(4))
\]
will be a non-trivial stably trivial spin torsor, which has no reduction of structure along the hyperbolic morphism $\op{SL}_2\to\op{Spin}(4)$. Moreover, by the previous remark, the Picard group for the schemes in examples~\ref{ex:so3:2} or \ref{ex:so3:3} is trivial, implying that the classification of spin torsors and quadratic forms agree in this case. Consequently, the torsor described above corresponds to a stably trivial (hence stably hyperbolic) quadratic form which is not hyperbolic. 
\end{example}


\begin{thebibliography}{GMV91}

\bibitem[ADF15]{AsokDoranFasel}
A. Asok, B. Doran and J. Fasel. Smooth models of motivic spheres and the clutching construction. Int. Math. Res. Not. IMRN, 2016. 

\bibitem[AF14a]{AsokFasel}
A. Asok and J. Fasel. A cohomological classification of vector bundles on smooth affine threefolds. Duke Math. J. 163 (2014), 2561--2601. 

\bibitem[AF14b]{AsokFaselSpheres}
A. Asok and J. Fasel. Algebraic vector bundles on spheres. J. Topology 7 (2014), 894--926.

\bibitem[AF15]{AsokFaselSplitting}
A. Asok and J. Fasel. Splitting vector bundles outside the stable range and $\mathbb{A}^1$-homotopy sheaves of punctured affine space. J. Amer. Math. Soc. 28 (2015), 1031--1062.

\bibitem[AHW15]{gbundles2}
A. Asok, M. Hoyois and M. Wendt. Affine representability results in $\mathbb{A}^1$-homotopy theory II: principal bundles and homogeneous spaces. Preprint, arXiv:1507.08020v1. 

\bibitem[AHW17]{octonion}
A. Asok, M. Hoyois and M. Wendt. Generically split octonion algebras and $\mathbb{A}^1$-homotopy theory. Preprint, arXiv:1704.03657v1, 2017.

\bibitem[AKW15]{AKW15}
A. Asok, S. Kebekus and M. Wendt. Comparing $\mathbb{A}^1$-h-cobordism and $\mathbb{A}^1$-weak equivalence. Preprint, arXiv:1410.3038, to appear in Ann. Sc. Norm. Super. Pisa. Cl. Sci.

\bibitem[BMS89]{BlochMurthySzpiro}
S. Bloch, M.P. Murthy and L. Szpiro. Zero cycles and the number of generators of an ideal. M{\'e}m. S.M.F. (2) 38 (1989), 51--74.

\bibitem[Gar15]{garrett}
P. Garrett. Sporadic isogenies to orthogonal groups. Notes, available online at \verb!http://www-users.math.umn.edu/~garrett/m/v/sporadic_isogenies.pdf!, 2015.

\bibitem[GJ09]{GoerssJardine}
P.G. Goerss and J.F. Jardine. \emph{Simplicial homotopy theory.} Modern Birkh\"auser Classics. Birkh\"auser Verlag, Basel, 2009. 

\bibitem[KM82]{KumarMurthy}
N.M. Kumar and M.P. Murthy. Algebraic cycles and vector bundles over affine three-folds. Ann. of Math. (2) 116 (1982), 579--591. 

\bibitem[Knu91]{knus}
M.A. Knus. Quadratic and Hermitian forms over rings. Grundlehren der Mathematischen Wissenschaften 294. Springer, 1991.

\bibitem[MK85]{mohan:kumar}
N. Mohan Kumar. Stably free modules. Amer. J. Math. 107 (1985), 1439--1443. 

\bibitem[Mor12]{MField}
F. Morel. \emph{$\mathbb{A}^1$-algebraic topology over a field}, volume 2052 of Lecture Notes in Mathematics. Springer, Heidelberg, 2012. 

\bibitem[MS76]{MurthySwan}
M.P. Murthy and R.G. Swan. Vector bundles over affine surfaces. Invent. Math. 36 (1976), 125--165.


\bibitem[Ste51]{steenrod}
N. Steenrod. The topology of fibre bundles. Princeton Mathematical Series,  14. Princeton University Press, 1951.

\bibitem[Swa85]{swan}
R.G. Swan. K-theory of quadric hypersurfaces. Annals of Math. 122 (1985), 113--153. 

\bibitem[Wen11]{torsors}
M. Wendt. Rationally trivial torsors in $\mathbb{A}^1$-homotopy theory. J. K-Theory 7 (2011), 541--572.

\bibitem[Wen17]{mk}
M. Wendt. Variations in $\mathbb{A}^1$ on a theme of Mohan Kumar. Preprint, arXiv:1704.00141v1.

\end{thebibliography}
\end{document}